\documentclass[11pt,a4paper,reqno]{article}
\usepackage[latin1]{inputenc}
\usepackage[T1]{fontenc}
\usepackage[english]{babel}
\usepackage{graphicx}
\usepackage{amsfonts,amssymb,amsbsy,latexsym,a4wide,xspace}
\usepackage{amsmath,delarray,enumerate,calc,amsthm}
\usepackage{bbm}
\usepackage{mathrsfs}
\usepackage{txfonts}
\usepackage{paralist}

\usepackage{authblk}

\usepackage{tikz}
\usetikzlibrary{matrix}

\usepackage{url}

\newcommand\blfootnote[1]{%
  \begingroup
  \renewcommand\thefootnote{}\footnote{#1}%
  \addtocounter{footnote}{-1}%
  \endgroup
}

\newcommand{\supp}{\operatorname{supp}}

\newcommand{\cA}{{\mathcal A}}
\newcommand{\cC}{{\mathcal C}}
\newcommand{\cB}{{\mathcal B}}
\newcommand{\cE}{{\mathcal E}}
\newcommand{\cF}{{\mathcal F}}
\newcommand{\cG}{{\mathcal G}}
\newcommand{\cH}{{\mathcal H}}

\newcommand{\cZ}{{\mathcal Z}}

\newcommand{\cP}{{\mathcal P}}

\newcommand{\cL}{{\mathcal L}}
\newcommand{\cM}{{\mathcal M}}

\newcommand{\cS}{{\mathcal S}}

\newcommand{\cU}{{\mathcal U}}

\newcommand{\N}{{\mathbbm N}}

\newcommand{\Z}{{\mathbbm Z}}

\newcommand{\1}{{\mathbbm 1}}

\newcommand{\inn}{\operatorname{int}}
\newcommand{\lcm}{\operatorname{lcm}}

\newcommand{\card}{\operatorname{card}}

\newcommand{\MWG}{\cM^{\scriptscriptstyle G}_{\scriptscriptstyle W}}

\newcommand{\oneone}{1-1\xspace}
\newcommand{\ddelta}{\boldsymbol{\delta}}
\newcommand{\prim}[1]{#1^{prim}}
\newcommand{\Ainf}{\cA_\infty}
\renewcommand{\mod}{\text{ mod }}
\newtheorem {definition}{Definition}[section]
\newtheorem {lemma}{Lemma}[section]
\newtheorem{theorem}{Theorem}
\renewcommand{\thetheorem}{\Alph{theorem}}
\newtheorem {bemerkung}{Remark}[section]
\newtheorem{proposition}{Proposition}[section]
\newtheorem {corollary}{Corollary}[section]
\newtheorem{beispiel}{Example}[section]
\newtheorem{frage}{Question}[section]
\newenvironment{question} {\begin{frage} \normalfont }{\end{frage}}
\newenvironment{remark} {\begin{bemerkung} \normalfont }{\end{bemerkung}}
\newenvironment{example} {\begin{beispiel} \normalfont }{\end{beispiel}}

\newcommand{\myitem}{\item[(\thetheorem\stepcounter{enumi}\theenumi)]}

\def\mapr#1#2{\smash{\mathop{\longrightarrow}\limits^{#1}_{#2}}}
\def\mapdown#1{\Big\downarrow\rlap{$\vcenter{\hbox{$\scriptstyle#1$}}$}}

\begin{document}

\title{Dynamics of $\cB$-free sets: a view through the window}
\author[1]{Stanis\l aw Kasjan$^{\ast\dagger}$}
\author[2]{Gerhard Keller$^\dagger$}
\author[1]{Mariusz Lema\'nczyk\thanks{Research supported by Narodowe Centrum Nauki  UMO-2014/15/B/ST1/03736.}\thanks{Research supported by the special program  of the semester
``Ergodic Theory and Dynamical Sytems in their Interactions with Arithmetic and Combinatorics'', Chair Jean Morlet, 1.08.2016-30.01.2017.}}
\affil[1]{\small Faculty of Mathematics and Computer Science, Nicolaus Copernicus University, Toru\'n, Poland}
\affil[2]{Department of Mathematics, University of Erlangen-N\"urnberg, Germany}
\date{Version of \today}

%%%%%%%%%%%%%%%%%%%%%%%%%%%%%%%%%%%%%%%%

\maketitle

{\begin{abstract}
Let $\cB$ be an infinite subset of $\{1,2,\dots\}$.
We characterize arithmetic and dynamical properties of the $\cB$-free set $\cF_\cB$ through
group theoretical, topological and measure theoretic properties of a set $W$ (called the \emph{window}) associated with $\cB$. This point of view stems from the interpretation of the set $\cF_\cB$ as a weak model set. Our main results are:  $\cB$ is taut if and only if the window is Haar regular; the dynamical system associated to $\cF_\cB$ is a Toeplitz system if and only if the window is topologically regular; the dynamical system associated to $\cF_\cB$ is proximal if and only if the window has empty interior; and the dynamical system associated to $\cF_\cB$ has the ``na\"ively expected'' maximal equicontinuous factor if and only if the interior of the window is aperiodic.
\end{abstract}}
\blfootnote{\emph{MSC 2010 clasification:} 37A35, 37A45, 37B05.}

\section{Introduction and main results}
For any given set $\cB\subseteq\N=\{1,2,\dots\}$ one can define its \emph{set of multiples}
\begin{equation*}
\cM_\cB:=\bigcup_{b\in\cB}b\Z
\end{equation*}
and the set of \emph{$\cB$-free numbers}
\begin{equation*}
\cF_\cB:=\Z\setminus\cM_\cB\ .
\end{equation*}
The investigation of structural properties of $\cM_\cB$ or, equivalently, of $\cF_\cB$ has a long history (see the monograph \cite{hall-book} and the recent paper \cite{BKKL2015} for references), and dynamical systems theory provides some useful tools for this. Namely, denote by $\eta\in\{0,1\}^\Z$ the characteristic function of $\cF_\cB$, i.e. $\eta(n)=1$ if and only if $n\in\cF_\cB$, and consider the
orbit closure $X_\eta$ of $\eta$ in the shift dynamical system
$(\{0,1\}^\Z,\sigma)$, where $\sigma$ stands for the left shift. Then topological dynamics and ergodic theory provide a wealth of concepts to describe various aspects of the structure of $\eta$, see \cite{Sa} which originated
 this point of view by studying the set of square-free numbers, and also
 \cite{Ab-Le-Ru}, \cite{BKKL2015}  which continued this line of research. 
 
 In this paper we continue to provide a dictionary that characterizes arithmetic properties of $\cB$ in terms of dynamical properties of $X_\eta$, and, as an intermediate step, also in terms of topological and measure theoretic properties of a pair $(H,W)$ associated with the passage from $\cB$ to
$X_\eta$, where $H$ is a compact abelian group and $W$ a compact subset of $H$. This latter point of view is borrowed from the theory of weak model sets, which applies here, because $\cF_\cB$ is a particular example of such a set, see e.g. \cite{BHS2015,KR2015}.
Finally the Chinese Remainder Theorem allows
 us to interpret our dynamical results combinatorially.

In order to formulate our main results, we need to
recall some notions from the theory of sets of multiples \cite{hall-book}
and also to
introduce some further notation.
Let $\cB$ be a non-empty subset of $\N$.
\begin{itemize}
\item $\cB$ is \emph{primitive}, if there are no $b,b'\in\cB$ with $b\mid b'$.
From any set $\cB\subseteq\N$ one can remove all multiples of other numbers in $\cB$, which results in the set
\begin{equation}\label{eq:Bprim}
\prim{\cB}:=\cB\setminus \bigcup_{b\in\cB}b\cdot(\N\setminus\{1\})\ .
\end{equation}
$\prim{\cB}$ is primitive by construction, and $\cM_\cB=\cM_{\prim{\cB}}$.
\item $\cB$ is \emph{taut}, if $\ddelta(\cM_{\cB\setminus\{b\}})<\ddelta(\cM_\cB)$ for each $b\in\cB$, where $\ddelta(\cM_\cB):=\lim_{n\to\infty}\frac{1}{\log n}\sum_{k\leqslant n,k\in\cM_\cB}k^{-1}$
denotes the logarithmic density of this set, which is known to exist by the Theorem of Davenport and Erd\"os \cite{DE1936,DE1951}.
So a set is taut, if removing any single point from it changes its set of multiples drastically and not only by ``a few points''.
\item $\tilde{H}:=\prod_{b\in\cB}\Z/b\Z$ and $\Delta:\Z\to\tilde{H}$, $\Delta(n)=(n,n,\dots)$ -- the canonical diagonal embedding.
\item $H:=\overline{\Delta(\Z)}$ is a compact abelian group, and we denote by $m_H$ its normalised Haar measure.
\item $R_{\Delta(1)}: H\to H$ denotes the rotation by $\Delta(1)$, i.e. $(R_{\Delta(1)}h)_b=(h_b+1)$ mod $b$ for all $b\in\cB$.
\item The \emph{window} is defined as
\begin{equation}\label{eq:W}
W:=\{h\in H: h_b\neq0\ (\forall b\in\cB)\}.
\end{equation}
\item $\varphi:H\to\{0,1\}^\Z$ is the coding function: $\varphi(h)(n)=1$, if  and only if $R_{\Delta(1)}^nh\in W$, equivalently, if and only if $h_b+n\neq0$ mod $b$ for all $b\in\cB$.
\item By $S,S'\subset\cB$ we always mean \emph{finite} subsets.
\item The topology on $H$ is generated by the (open and closed) cylinder sets
\begin{equation*}
U_S(h):=\{h'\in H: \forall b\in S: h_b=h'_b\},\text{ defined for finite }S\subset\cB\text{ and }h\in H\ .
\end{equation*}
\end{itemize}
A recurring theme of the main  results  in this paper is to characterize arithmetic and dynamical properties of a $\cB$-free set $\cF_\cB$ through
group theoretical, topological and measure theoretic properties of the window $W$ defined above.
\begin{remark}
With the notation introduced above, we can write
\begin{equation*}
X_\eta=\overline{\varphi(\Delta(\Z))}\ .
\end{equation*}
This is certainly a subset of $X_\varphi:=\overline{\varphi(H)}$, the set studied in \cite{KR2015} under the name $\MWG$. In Proposition~\ref{prop:Mariusz4.2} we show that $X_\eta=X_\varphi$ when $\cB$ has \emph{light tails} (see Subsection~\ref{subsec:light-tails} for a definition), but we do not know whether also tautness of $\cB$ suffices (see also Subsection~\ref{subsec:light-tails}).
\end{remark}

\subsection{Tautness as a measure theoretic property}

\begin{theorem}\label{theo:taut-regular}
\footnote{The authors are indebted to J. Ku\l{}aga-Przymus for pointing out the relevance of \cite[Lemma 1.17]{hall-book} for the proof of this theorem.}
Suppose that the set $\cB$ is primitive.
Then the following are equivalent:
\begin{compactenum}[(i)]
\item $\cB$ is taut.
\item The window $W$ is \emph{Haar regular}, i.e. $\supp(m_H|_W)=W$.
\end{compactenum}
Moreover, these properties imply
\begin{compactenum}[(i)]
\addtocounter{enumi}{2}
\item $\overline{\Delta(\Z)\cap W}=W$.
\end{compactenum}
\end{theorem}
\noindent
The proof of the theorem is provided in Section~\ref{sec:proof-taut-regular}.
The concept of a \emph{Haar regular window} was introduced in \cite{KR2016} in the context of general weak model sets.
\\[0mm]

Given a set $\cB\subset\N$, one says that $h:=(h_b)_{b\in\cB}\in\Z^{\cB}$ satisfies the CRT
(Chinese Remainder Theorem) if for each finite $S\subset\cB$  there exists $n\in\Z$ such that
\begin{equation}\label{freeCRT}
h_b=n\text{ mod }b\;\text{ for each }b\in S.
\end{equation}
Clearly,
\begin{equation*}
h\text{ satisfies the CRT\; iff\;} h\in H.
\end{equation*}

We are looking for solutions of~\eqref{freeCRT} with $n\in\cF_\cB$. If for $h$ as above we can solve~\eqref{freeCRT} with $n=n_S\in\cF_\cB$ for all finite $S\subset \cB$, then we say that $h$ satisfies the $\cB$-free CRT.
A necessary condition for $h=(h_b)_{b\in\cB}$ to satisfy the $\cB$-free CRT is, of course, that $h_b\neq0\mod b$ for each $b\in\cB$, and a moment's reflection shows that
\begin{equation}\label{eq:B-free-CRT}
h\text{ satisfies the $\cB$-free CRT\; iff\;} h\in \overline{\Delta(\Z)\cap W}\ .
\end{equation}
Therefore the implication $(i)$ $\Rightarrow$ $(iii)$ of Theorem~\ref{theo:taut-regular} is an immediate consequence of the following proposition.
\begin{proposition}\label{prop:mariusz-stronger}
Assume that $\cB$ is taut. Let $h\in W$ and $S\subset\cB$ finite.
Then  the set of
\emph{$\cB$-free} integers $n$ that solve $n = h_b$ mod $b$ for $b\in S$ has asymptotic density
$m_H(U_S(h)\cap W)>0$.
\end{proposition}

In Subsection~\ref{subsec:property(iii)} we provide a sequence $\cB$, which is not taut, but for which $\overline{\Delta(\Z)\cap W}=W$ (Example~\ref{example:non(iii)to(i)}). Hence $(iii)$ of Theorem~\ref{theo:taut-regular} is not equivalent to $(i)$ and $(ii)$. Here we provide two simpler examples which throw some light on property $(iii)$. Denote by $\cP\subseteq\N$ the set of all \emph{prime numbers}.

\begin{example} If $\cB=\cP$ then $H=\prod_{p\in\cP}\Z/p\Z$, $W$ is uncountable (although of Haar measure zero) and
$\overline{\Delta(\Z)\cap W}\neq W$,
since for each $n$ we find $p\in\cP$ such that $p\mid n$, so $n=0$ mod~$p$.
\end{example}

\begin{example} If $\cB\subset{\cP}$ is \emph{thin}, i.e. if $\sum_{p\in\cB}1/p<+\infty$, then $\overline{\Delta(\Z)\cap W}=W$ in view of \eqref{eq:B-free-CRT},
because each $h\in H$ satisfies the $\cB$-free CRT.  Indeed, if $S\subset\cB$ is finite and
$n=h_b\mod b$ for $b\in S$, then  $n+\lcm(S)\Z$ is the set of all solutions to this system of congruences. Moreover, if $h\in W$, then $\gcd(n,b)=1$ for all $b\in S$. We only need to find $r\in\Z$ so that $n+r\lcm(S)$ is a prime number which is not in $\cB$. The latter follows from Dirichlet's theorem:
The set of prime numbers contained in $n+\lcm(S)\Z$ is not thin. Of course this is a special case of Theorem~\ref{theo:taut-regular}.
\end{example}

\begin{remark}
Denote by $\nu_\eta:=m_H\circ\varphi^{-1}$ the \emph{Mirsky measure} on $X_\eta$.
There are two independent proofs of the fact that the two equivalent conditions from Theorem~\ref{theo:taut-regular} imply that
the measure preserving dynamical system $(X_\eta,\sigma,\nu_\eta)$ is isomorphic to the group rotation $(H,R_{\Delta(1)},m_H)$:
 In \cite[Theorem F]{BKKL2015} it is proved that this is implied by $(i)$. That it is also a direct consequence of $(ii)$ follows - in the more general context of model sets - from \cite{KR2016}. The proof uses our observation that $W$ is aperiodic (see Proposition~\ref{p:ape3}). To see this, denote by $H_W:=\{h\in H: W+h=W\}$ the period group of $W$ and by $H_W^{Haar}:=\{h\in H: m_H((W+h)\triangle W)=0\}$ its group of \emph{Haar periods}. It is easily seen that
$H_W=H_W^{Haar}$ for Haar regular $W$, in particular whenever the sequence $\cB$ is taut. Hence, if $W$ is aperiodic, it is also Haar aperiodic, and this is what is needed to apply the general theorem from \cite{KR2016} to the present context.

A word of caution is in order at this point: Althoug, in the $\cB$-free context, the window $W$ is always aperiodic (Proposition~\ref{p:ape3}), this is not necessarily true for its Haar regularization $W_{reg}:=\supp(m_H|_W)$, because that window is not of the same arithmetic type as $W$. On the other hand, as proved in \cite[Theorem C]{BKKL2015}, each non-taut set $\cB$ can be modified into a taut set $\cB'$ whose corresponding Mirsky measure $\nu_{\eta'}$ coincides with $\nu_\eta$ (as a measure on $\{0,1\}^\Z$).
The (arithmetic!) window $W'\subseteq H'$ defined by $\cB'$ is then aperiodic and Haar regular, and we suspect that it to be
closely related to $W_{reg}\subseteq H$.
\end{remark}

\subsection{The proximal and the Toeplitz case}
From  \cite[Theorem A]{BKKL2015} we know that $X_\eta$ has a unique minimal subset $M$.
In Lemma~\ref{lemma:C_varphi}, we prove that  $M=\overline{\varphi(C_\varphi)}$, where $C_\varphi$ denotes the set of continuity points of $\varphi:H\to\{0,1\}^\Z$, see also \cite[Lemma 6.3]{KR2015}.
$M$ is degenerate to a singleton, namely to $M= \{(\dots,0,0,0,\dots)\}$, if and only if $\inn(W)=\emptyset$  \cite{KR2015}, and we collect a number of equivalent characterizations of this extreme case in Theorem~\ref{theo:proximal} below.
Assuming primitivity of $\cB$ and property $(iii)$ of Theorem~\ref{theo:taut-regular}, we prove the following equivalent characterizations of minimality of $(X_\eta,\sigma)$, i.e. of  $M=X_\eta$,  in Subsection~\ref{subsec:proof-theo-minimality}.
For $S\subset\cB$ let
\begin{equation}\label{eq:A_S-def}
\cA_S:=\{\gcd(b,\lcm(S)): b\in\cB\},
\end{equation}
and note that $\cF_{\cA_S}\subseteq\cF_\cB$, because
$b\mid m$ for some $b\in\cB$ implies $\gcd(b,\lcm(S))\mid m$ for any $S\subset\cB$.
Let
\begin{equation}\label{def:Cinf}
\Ainf:=\{n\in\N: \forall_{S\subset\cB}\ \exists_{S': S\subseteq S'}: n\in\cA_{S'}\setminus S'\}.
\end{equation}
In Lemma~\ref{lemma:ASk-C} we prove: If $(S_k)_k$ is a filtration of $\cB$ with finite sets, then
\begin{equation}
\limsup_{k\to\infty}\left(\cA_{S_k}\setminus S_k\right)=\Ainf\ .
\end{equation}

\begin{theorem}\label{theo:minimality}
Suppose that $\cB$ is primitive.
Consider the following list of properties:
%Assume that the set $\cB$ is primitive and $\overline{\Delta(\Z)\cap W}=W$  (which is always true if the set $\cB$ is taut, see Theorem~\ref{theo:taut-regular}). Then the following are equivalent.
\begin{compactenum}
\myitem The window $W$ is topologically regular, i.e. $\overline{\inn(W)}=W$.
\myitem  $\cF_\cB=\bigcup_{S\subset\cB\text{ finite}}\cF_{\cA_S}$.
\myitem  $\Ainf=\emptyset$.
\myitem There are no $d\in\N$ and no infinite pairwise
coprime set $\cA\subseteq\N\setminus\{1\}$ such that $d\,\cA\subseteq\cB$.
\myitem $\eta=\varphi(0)$ is a Toeplitz  sequence (see \cite{Do}, \cite{JK1969} for the
definition) different from $(\dots,0,0,0,\dots)$.
\myitem $0\in C_\varphi$  and
$\varphi(0)\neq(\dots,0,0,0,\dots)$.
\myitem $\eta\in M$ and $\eta\neq(\dots,0,0,0,\dots)$.
\myitem $X_\eta$ is minimal., i.e. $X_\eta=M$, and $\card(X_\eta)>1$.
\myitem The dynamics on $X_\eta$
is a minimal almost \oneone extension of $(H,R_{\Delta(1)})$,
the rotation by $\Delta(1)$ on $H$.
\end{compactenum}
\begin{compactenum}[a)]
\item (B1) - (B6) are all equivalent,
and each of these conditions implies that $\cB$ is taut.
\item (B7) and (B8) are equivalent.
\item Each of (B1) - (B6) implies (B9).
\item (B9) implies (B7) and (B8).
\item If $\overline{\Delta(\Z)\cap W}=W$  (in particular if $\cB$ is taut), then (B1) - (B9) are all equivalent.
\end{compactenum}
\end{theorem}

\begin{remark}\label{remark:window-toeplitz}
%A topologically regular window is in particular Haar regular, because $\inn(W)$ is always contained in the support of $m_H|_W$. Hence, if $W$ is topologically regular and $\cB$ is primitive, then $\cB$ is taut (Theorem~\ref{theo:taut-regular}), and
%Theorem~\ref{theo:minimality} shows that $\eta$ is a Toeplitz sequence.
One ingredient of the proof of Theorem~\ref{theo:minimality} is the observation that the set $\cB$ is taut whenever $\eta$ is a Toeplitz sequence. This was pointed out to us by
A. Bartnicka  who also gave a proof of it, which we recall in Lemma~\ref{lemma:Aurelia}  below.

Moreover, we can interpret the result purely arithmetically as follows:  If $\cB$ is primitive and satisfies (B4) then the set of elements for which the $\cB$-free CRT  holds is topologically regular, i.e. it contains a dense subset of points for which all sufficiently close points satisfying the CRT satisfy also the $\cB$-free CRT.
\end{remark}

The following characterization of regular Toeplitz sequences is included in Proposition~\ref{prop:toeplitz} in
Subsection~\ref{subsec:Toeplitz}, where also the precise definition of regularity
of a Toeplitz sequence is recalled.
\begin{proposition}\label{prop:Toeplitz-regular}
Assume that $\overline{\inn(W)}=W$. Then the Toeplitz sequence
$\eta$ is regular, if and only if $m_H(\partial W)=0$.
\end{proposition}
In Subsection~\ref{subsec:Toeplitz} we also provide examples of sets $\cB$ that give rise to regular Toeplitz sequences and others giving rise to irregular Toeplitz sequences.
Note also that
$m_H(\partial W)=0$ if and only if
$\inf_{S\subset\cB}\bar d(\cM_{\cA_S}\setminus\cM_\cB)=0$, see
Lemma~\ref{lemma:Haar-boundary}, and observe that $m_H(\partial W)=0$ implies unique ergodicity of the dynamics on $X_\eta$ \cite[Theorem 2c]{KR2015}.
\quad\\

The next theorem is complementary to Theorem~\ref{theo:minimality}.
Most of its equivalences follow from results in \cite{BKKL2015} and \cite{KR2015} and are proved in Subsection~\ref{subsec:proof-theo-proximal}.
They do not rely on the more advanced arithmetic concept of tautness.
\begin{theorem}\label{theo:proximal}
The following are equivalent:
\begin{compactenum}
\myitem $\inn(W)=\emptyset$
\myitem  $\bigcup_{S\subset\cB\text{ finite}}\cF_{\cA_S}=\emptyset$,
i.e. $\cF_{\cA_S}=\emptyset$ for all finite $S\subset\cB$.
\myitem $\forall S\subset\cB: 1\in\cA_S$.
\myitem $\cB$ contains an infinite pairwise coprime subset.
\myitem If $\cC\subseteq\N$ is finite and if $\cB\subseteq\cM_\cC$, then $1\in\cC$.
\myitem $M=\{(\dots,0,0,0,\dots)\}$.
\myitem The dynamics on $X_\eta$ are proximal.
\end{compactenum}
\end{theorem}

\begin{remark}
Under the conditions
of Theorem \ref{theo:proximal}  no element  of $W$ is
stable, that is, for each $h$ staisfying the $\cB$-free CRT there is an element
  $h'\in\Z^{\cB}$ arbitrarily close to $h$ which satisfies the CRT but not the $\cB$-free
 CRT.\end{remark}

\subsection{The maximal equicontinuous factor}
We finish with a result that identifies the maximal equicontinuous factor of the dynamics on $X_\eta$ and answers Question 3.14 in \cite{BKKL2015}.
Given a subset $A\subseteq H$, denote by
\begin{equation*}
H_A:=\left\{h\in H: A+h=A\right\}
\end{equation*}
the \emph{period group} of $A$. The set $A\subseteq H$ is \emph{topologically aperiodic}, if $H_A=\{0\}$.
Observe also that $H_{\inn(A)}$ is a closed subgroup of $H$, whenever $A$ is closed \cite[Lemma 6.1]{KR2016}.
%
%The following simple observations are proved in \cite{KR2016}:
%\begin{itemize}
%\item $H_A\subseteq H_{\bar A}\cap H_{\inn(A)}$.
%\item If $A$ is closed, then $H_{\inn(A)}=H_{\overline{\inn(A)}}$ is closed.
%\end{itemize}

In Proposition~\ref{p:ape3} we prove that $H_W=\{0\}$ whenever $\cB$ is primitive.
If $\inn(W)=\emptyset$, then of course $H_{\inn(W)}=H$. If $\inn(W)\neq\emptyset$,
the situation is more complicated: $H_{\inn(W)}$ is obviously always a strict subgroup of $H$, and very often $H_{\inn(W)}=\{0\}$, but
there are examples where $H_{\inn(W)}$ is a non-trivial strict subgroup of $H$, see Subsection~\ref{subsec:examples}. In any case, however, $H_{\inn(W)}$ determines the maximal equicontinuous factor. The following is proved in \cite[Theorem A2]{KR2016}:
\begin{equation}
\begin{minipage}{0.9\textwidth}
\textbf{Theorem}\; \emph{The translation by $\Delta(1)+H_{\inn(W)}$ on $H/H_{\inn(W)}$ is the maximal equicontinuous factor of the dynamics on $X_\eta$.}
\end{minipage}
\end{equation}
Let $S_1\subset S_2\subset\dots$ be any filtration of $\cB$ by finite sets.
In Subsection~\ref{subsec:period-group} we define divisors $d_k$ of $\lcm(S_k)$:
\begin{equation}
d_k:=\lim_{j\to\infty}\gcd(s_k,c_{k+j})\,,\;\text{ where }\;
s_k:=\lcm(S_k)\;\text{ and }\;
c_l:=\text{minimal period of }\cM_{\cA_{S_l}}\ .
\end{equation}
By Remark \ref{remark:factors} we have $\frac{s_k}{d_k}\mid\frac{s_{k+1}}{d_{k+1}}$ for any $k$.
The sequences $(s_k)$, $(d_k)$ and $(c_k)$ determine $H_{\inn(W)}$ in the following way:
\begin{proposition}\label{prop:period}
\begin{compactenum}[a)]
\item
\begin{equation*}
0\rightarrow H_{\inn(W)}\rightarrow H\cong\lim\limits_{\leftarrow}\Z/{s_k}\Z\rightarrow \lim\limits_{\leftarrow}\Z/{d_k}\Z\rightarrow 0
\end{equation*}
is an exact sequence.\footnote{A sequence of abelian groups and homomorphisms  $...\mapr{}{} M_{k-1}\mapr{f_{k-1}}{} M_{k}\mapr{f_k}{} M_{k+1}\rightarrow...$ is called {\em exact}
if the kernel of  $f_k$ is equal to the image of $f_{k-1}$ for any $k$. In particular, a sequence
$$
0\rightarrow M' \mapr{f}{} M\mapr{g}{} M''\rightarrow 0
$$
is exact, when $f$ is injective, the kernel of $g$ equals the image of $f$ and $g$ is surjective. We say that it is a "short exact sequence". In particular, the homomorphism $g$ induces an isomorphism $M''\cong M/f(M')$ in this case.
}
\item $H_{\inn(W)}\cong \lim\limits_{\leftarrow}\Z/\frac{s_k}{d_k}\Z$.
\item $H/H_{\inn(W)}\cong \lim\limits_{\leftarrow}\Z/{d_k}\Z$.
\item $H_{\inn(W)}=\{0\}$ if and only if $s_k=d_k$ for each $k\in\N$,
equivalently if for each $b\in\cB$ there is $n>0$ such that $b$ divides $c_{n}$.
\end{compactenum}
\end{proposition}

\begin{theorem}\label{theo:MEF}
\begin{compactenum}[a)]
\item The translation by $(1,1,\dots)$ on $H/H_{\inn(W)}\cong\lim\limits_{\leftarrow}\Z/{d_k}\Z$ is the
maximal equicontinuous factor of the dynamics on $X_\eta$.
\item In case d) of Proposition~\ref{prop:period},
the translation by $\Delta(1)$ on $H\cong \lim\limits_{\leftarrow}\Z/{s_k}\Z$ is the maximal equicontinuous factor of the dynamics on $X_\eta$.
\end{compactenum}
\end{theorem}
In Subsection~\ref{subsec:examples} we provide a number of examples illustrating this theorem.

\begin{remark}\label{rem:Y}
In \cite{BKKL2015}, the following set $Y$ is defined: \footnote{Versions of this set occur also in \cite{Peckner2012} and \cite{BaakeHuck14}.}
\begin{equation*}
Y:=\left\{x\in\{0,1\}^\Z: \card(\supp(x)\mod b)=b-1\;\forall b\in\cB\right\}.
\end{equation*}
Observe that $\card(\supp(x)\mod b)\leqslant b-1$ for all $x\in X_\eta$ and $b\in\cB$.
\footnote{Indeed, if $\card(\supp(x)\mod b)= b$ for some $x\in X_\eta$ and $b\in\cB$,
then this happens on some integer interval $[-M,M]$, and hence
$\card(\supp(\eta)\mod b)= b$, which contradicts the fact that $\supp(\eta)\subseteq\cF_\cB$.\label{foot:b-1}}
Proposition 3.27 of \cite{BKKL2015} asserts that
$(H,R_{\Delta(1)})$ is the maximal equicontinuous factor of $(X_\eta,S)$,
whenever $X_\eta\subseteq Y$. Hence, in that case, $H_{\overline{\inn W}}= H_{\inn W}=\{0\}=H_W$ by Theorem~\ref{theo:MEF} and Proposition~\ref{p:ape3}. This is the second one of the following two implications:
\begin{equation}
W=\overline{\inn W}\quad\Rightarrow\quad X_\eta\subseteq Y\quad\Rightarrow\quad
H_W=H_{\overline{\inn W}}\ .
\end{equation}
The first one is proved in Proposition~\ref{eta_in_Y}.
\end{remark}

\section{Tautness of $\cB$ and Haar regularity of $W$}\label{sec:proof-taut-regular}
\subsection{Arithmetic of $\cB$ and topology of $W$, part I}

\begin{definition}
Let $\cM\subseteq\N$.
\begin{compactenum}[a)]
\item The \emph{upper} resp. \emph{lower density} of $\cM$ is
\begin{equation*}
\overline{d}(\cM)=\limsup_{N\to\infty}
\frac1N\card\left(\cM\cap\{1,\dots,N\}\right)
\text{ resp. }
\underline{d}(\cM)=\liminf_{N\to\infty}
\frac1N\card\left(\cM\cap\{1,\dots,N\}\right)
\end{equation*}
If the limit exists, we write $d(\cM)$.
\item The \emph{logarithmic density} of $\cM$ is
\begin{equation*}
\ddelta(\cM)=\lim_{N\to\infty}\frac{1}{\log N}\sum_{n\in\cM\cap\{1,\dots,N\}}\frac{1}{n}
\end{equation*}
whenever the limit exists.
\end{compactenum}
\end{definition}
The theorem of Davenport and Erd\"os \cite{DE1936,DE1951} asserts that $\ddelta(\cM_\cB)=\underline{d}(\cM_\cB)$ exists for any subset $\cB\subseteq\N$.
\begin{definition}
$\cB\subseteq\N\setminus\{1\}$ is a \emph{Behrend sequence}, if $\ddelta(\cM_\cB)=1$.
\end{definition}
Recall that $\cB$ is {taut}, if $\ddelta(\cM_{\cB\setminus\{b\}})<\ddelta(\cM_\cB)$ for each $b\in\cB$.
The following is a corollary to a theorem of Behrend \cite{Behrend1948}:
\begin{proposition}
A set $\cB\subseteq\N$ is taut, if and only if it is primitive and there are no $q\in\N$ and no
Behrend set $\cA\subseteq \N\setminus\{1\}$ such that
$q\,\cA\subseteq\cB$ \cite[Corollary 0.19]{hall-book}.
\end{proposition}
This motivates the next definition:

\begin{definition}\label{def:pre-taut}
A set $\cB\subseteq\N$ is \emph{pre-taut}, if there are no $	q\in\N$ and Behrend set $\cA\subseteq\N\setminus\{1\}$ such that $q\,\cA\subseteq\cB$.
\end{definition}

\begin{lemma}\label{lemma:taut-facts}
Let $\cB\subseteq\N$ and $c\in\N$.

\begin{compactenum}[a)]
\item If $c\,\cB$ is pre-taut, then also $\cB$ is pre-taut. Moreover,
$\cB$ is
taut if and only if $c\,\cB$ is taut.
\item Each subset of a (pre-)taut set is (pre-)taut.
\item A finite union of pre-taut sets is pre-taut.
\item If $\cB$ is taut, then $\cB=\{1\}$ or $d(\cM_\cB)\neq1$ (possibly non-existing). Equivalently, if $d(\cM_\cB)=1$, then $\cB=\{1\}$ or $\cB$ is not taut.
\item If $\cB$ is pre-taut, then $1\in\cB$ or $d(\cM_\cB)\neq1$ (possibly non-existing). Equivalently, if $d(\cM_\cB)=1$, then $1\in\cB$ or $\cB$ is not pre-taut.
\end{compactenum}
\end{lemma}

\begin{proof}
a)\;The first implication is obvious. It is also clear that $\cB$ is primitive if and only if $c\cB$ is primitive.  Moreover,
\begin{equation*}
\begin{split}
\cB\text{ is taut}
&\Leftrightarrow
\forall b\in\cB: \underline{d}(\cM_\cB)>\underline{d}(\cM_{\cB\setminus\{b\}})
\Leftrightarrow
\forall b\in\cB: c^{-1}\underline{d}(\cM_\cB)>c^{-1}\underline{d}(\cM_{\cB\setminus\{b\}})\\
&\Leftrightarrow
\forall b\in\cB: \underline{d}(\cM_{c\cB})>\underline{d}(\cM_{c(\cB\setminus\{b\})})
=\underline{d}(\cM_{c\,\cB\setminus\{cb\}})\\
&\Leftrightarrow
\forall b'\in c\,\cB: \underline{d}(\cM_{c\cB})>\underline{d}(\cM_{c\,\cB\setminus\{b'\}})\\
&\Leftrightarrow
c\,\cB\text{ is taut.}
\end{split}
\end{equation*}
{b) is obvious (see Remark 2.1)}.
\\
{c) follows from \cite[Corollary 0.14]{hall-book}, see also \cite[Proposition 2.33]{BKKL2015}}.
\\
d) Suppose that $\cB$ is taut. Then $\cB$ is primitive, and $d(\cM_\cB)\neq1$ unless $1\in\cB$ by \cite[Corollary 0.19]{hall-book}.
Hence $d(\cM_\cB)\neq1$ or $\cB=\{1\}$.\\
e) follows directly from Definition~\ref{def:pre-taut} \footnote{{Note that d) follows from e) and Remark 2.2}.}.
\end{proof}

\begin{remark}
$\cB$ is taut if and only if it is pre-taut and primitive. If $\cB$ is pre-taut, then $\prim{\cB}$ is taut in view of Lemma~\ref{lemma:taut-facts}b.
\end{remark}

For $q\in\N$ and $\cB\subseteq\Z$ let
\begin{equation*}
\cB'(q)=\left\{\frac{b}{\gcd(b,q)}:b\in\cB\right\},
\end{equation*}
and note that $1\in\cB'(q)$ if and only if $q\in\cM_\cB$.

\begin{lemma}\label{lemma:introductory}
Let $q\in\N$, $\cB,\cC\subseteq\Z$, and $q\,\cC\subseteq\cM_\cB$. Then
$\cM_\cC\subseteq\cM_{\cB'(q)}$.
\end{lemma}

\begin{proof}
Let $c\in \cC$. There are $\ell\in\Z$ and $b\in\cB$ such that $qc=\ell b$.
Since $q\mid\ell b$, it follows that $q\mid\ell \gcd(b,q)$, thus $k=\frac{\ell \gcd(b,q)}{q}$ is an integer. We have
\begin{equation*}
c=
\frac{\ell b}{q}
=
k\cdot\frac{b}{\gcd(b,q)}
\in\cM_{\cB'(q)}\ .
\end{equation*}
This shows that $\cC\subseteq\cM_{\cB'(q)}$ and hence also $\cM_\cC\subseteq\cM_{\cB'(q)}$.
\end{proof}

\begin{lemma}\label{lemma:q-taut}
Let $\cB\subseteq\N$ and $q\in\N$.
\begin{compactenum}[a)]
\item If $\cB$ is pre-taut, then  $\cB'(q)$ is pre-taut.
\item If $\cB$ is taut, then  $\cB'(q)$ is a finite disjoint union of taut sets $\cB_i'$ defined below in the proof {of a)}.
\item If $d(\cM_\cB)=1$, then $d(\cM_{\cB'(q)})=1$.

\item If $\cB=\bigcup_{i=1}^N\cC_i$ and if $d(\cM_\cB)=1$, then $d(\cM_{\cC_i})=1$ for at least one $i\in\{1,\dots,N\}$.
\end{compactenum}
\end{lemma}
\begin{proof}
Let $I:=\left\{\frac{q}{\gcd(b,q)}:b\in\cB\right\}$. For $i\in I$
denote $\cB_i:=\left\{b\in\cB:\frac{q}{\gcd(b,q)}=i\right\}$ and
$\cB_i':=\left\{\frac{b}{\gcd(b,q)}: b\in\cB_i\right\}$. Then $I$ is finite, $\cB=\bigcup_{i\in I}\cB_i$, and $\cB'(q)=\bigcup_{i\in I}\cB_i'$. Moreover,
$\cB_i=\left\{\frac{q}{i}b':b'\in\cB_i'\right\}=\frac{q}{i}\cB_i'$.\\

a)\; If $\cB$ is pre-taut, then all $\cB_i$ are pre-taut (Lemma~\ref{lemma:taut-facts}b), then all
$\cB_i'$ are pre-taut (Lemma~\ref{lemma:taut-facts}a), and then $\cB'(q)$ is pre-taut (Lemma~\ref{lemma:taut-facts}c).\\
b)\;If $\cB$ is taut, then all $\cB_i$ are taut (Lemma~\ref{lemma:taut-facts}b), and then all
$\cB_i'$ are taut (Lemma~\ref{lemma:taut-facts}a).\\
c)\;As $\cB\subseteq\cM_{\cB'(q)}$, we have also $\cM_\cB\subseteq\cM_{\cB'(q)}$.\\
d)\;If $\cB$ is Behrend, then at least one of the sets $\cC_i$ is Behrend \cite[Corollary 0.14]{hall-book}, and so ${d(\cM_{\cC_i})=1}$.
Otherwise $1\in\cB$, so that $1\in\cC_i$ for some $i$, whence $\cM_{\cC_i}=\Z$.
\end{proof}

\begin{lemma}\label{lemma:gen_of_4.25}(compare \cite[Proposition 4.25]{BKKL2015})
Assume that $\cB\subseteq\N$ is taut and $d(\cM_\cC)=1$ for some $\cC\subseteq\Z$. If
$q\,\cC\subseteq\cM_\cB$ for some $q\geqslant1$,
then $b\mid q$ for some $b\in\cB$.
\end{lemma}
\begin{proof}
By Lemma~\ref{lemma:introductory}, $\cM_\cC\subseteq\cM_{\cB'(q)}$, so that $d(\cM_{\cB'(q)})=1$.
Then $\cB'(q)=\{1\}$ or $\cB'(q)$ is not taut (Lemma~\ref{lemma:taut-facts}d). If $\cB'(q)=\{1\}$, then $\gcd(b,q)=b$ for all $b\in\cB$, i.e $b\mid q$ for all $b\in\cB$, which is impossible because $\cB$ is infinite. Hence $\cB'(q)$ is not taut.
On the other hand, as $\cB$ is taut by assumption, $\cB'(q)$ is a finite union of taut sets $\cB_i'$ (Lemma~\ref{lemma:q-taut}b). As $d(\cM_{\cB'(q)})=1$, also $d(\cM_{\cB_i'})=1$ for at least one of the sets $\cB_i'$ (Lemma~\ref{lemma:q-taut}d), so that $\cB_i'=\{1\}$ for this set (Lemma~\ref{lemma:taut-facts}d). This implies $q\in\cM_\cB$.
\end{proof}

Recall that the topology on $H$ is generated by the (open and closed) cylinder sets
\begin{equation*}
U_S(h):=\{h'\in H: \forall b\in S: h_b=h'_b\},\text{ defined for finite }S\subset\cB\text{ and }h\in H\ ,
\end{equation*}
and recall also the definition of $\cA_S:=\{\gcd(b,\lcm(S)):b\in\cB\}$. Note that $\cA_S$ is finite and  $S\subseteq\cA_\cS$.
\begin{lemma}\label{lemma:easy}
Let $U=U_S(\Delta(n))$ for some $S\subset\cB$ and $n\in\Z$.
\begin{compactenum}[a)]
\item If $n\in\cM_S$, then $U\cap W=\emptyset$.
\item If $U\cap W=\emptyset$, then $n+\lcm(S)\cdot\Z\subseteq \cM_{\cB\cap\cA_S}$.
\item There is a filtration of $\cB$ by finite sets $S$ for which $\cB\cap\cA_S= S$.
\item If $\cB\cap\cA_S= S$, then $n\in\cM_S$ iff $U\cap W=\emptyset$
iff $n+\lcm(S)\cdot\Z\subseteq\cM_S$.
\end{compactenum}
\end{lemma}

\begin{proof}
a)\; This follows immediately from the definitions of $U_S(\Delta(n))$ and $W$.\\
b)\; For each $h\in U$ there is $b\in\cB$ such that $h_b=0$.
As $U$ is compact, the Heine-Borel argument produces
a finite set $S'\subset\cB$ such that for each $h\in U$
there is $b\in S'$ such that $h_b=0$.
Let $s=\lcm(S)$. This observation applies in particular to all $h\in\Delta(n+s\Z)\subseteq U_{S}(\Delta(n))=U$. That means, for each $k\in\Z$ there is $b_k\in S'$ such that
$b_k\mid n+sk$. In other words: $n+s\Z\subset\cM_{S'}$.
The set $S'$ need not be primitive automatically, but we can replace it w.l.o.g. by a primitive subset without changing its set of multiples. Then,
as $S'$ is finite, it is taut. Denote $q=\gcd(n,s)$ and $\cC=\frac{n}{q}+\frac{s}{q}\cdot\Z$. Then $q\,\cC=n+s\Z\subseteq\cM_{S'}$, and as
$\gcd(n/q,s/q)=1$, $d(\cM_\cC)=1$ (Dirichlet, see \cite[Corollary 4.24]{BKKL2015}).
Now Lemma~\ref{lemma:gen_of_4.25} shows that $b\mid q=\gcd(n,s)$ for some $b\in S'$. In particular, $n\in b\Z$
and $b\mid s=\lcm(S)$ for that $b\in S'$, so that $b=\gcd(b,s)\in\cB\cap\cA_S$ and
$n+\lcm(S)\cdot\Z\subseteq b\Z\subseteq\cM_{\cB\cap\cA_S}$.\\
c)\; It suffices to prove that for any finite $S\subset\cB$ there exists a finite $S'\subset\cB$ with $\cB\cap\cA_{S'}=S'$. So let $S\subset\cB$ and $S':=\cB\cap\cA_S$. $S'$ is finite, because $\cA_S$ is finite, and obviously $S\subseteq S'\subseteq\cB\cap\cA_{S'}$. As each $b'\in S'\subseteq\cA_S$ divides $\lcm(S)$, also $\lcm(S')$ divides $\lcm(S)$. Therefore $\lcm(S')=\lcm(S)$, so that  $\cA_{S'}=\cA_S$.
Hence $S'=\cB\cap\cA_{S'}$.
\\
d)\; This follows from a) and b).
\end{proof}

\subsection{Proof of Theorem~\ref{theo:taut-regular}}
Let $\cB=\{b_1,b_2,\dots\}$ be primitive, and denote $S_1\subset S_2\subset\dots\subset\cB$ a filtration of $\cB$ by finite sets $S_k$.
Let $s_k=\lcm(S_k)$.
We can assume without loss of generality that $b\mid s_k\Rightarrow b\in S_k$ holds for all $b\in\cB$ and all $k\in\N$.
For each $k\in\N$, the collection of all cylinder sets $U_{S_k}(h)$, $h\in H$, can be written explicitly as
\begin{equation*}
\cZ_k:=\left\{U_{S_k}(\Delta(n)): n=1,\dots,s_k\right\}.
\end{equation*}

Suppose first that $\cB$ is not taut. Then it contains a scaled copy $c\cA$ of a Behrend set $\cA\subseteq\{2,3,\dots\}$.
Enlarging $\cA$, if necessary, we can assume that $c\cA=\cB\cap c\Z$. (As $\cB$ is primitive, also the enlarged $\cA$ does not contain the number $1$.)
 Let $a_0>1$ be the smallest element of $\cA$ and denote $b_0=ca_0$. Let $H_0=\{h\in H:h_{b_0}\in c\Z\}$. Then $H_0$ is open and closed, and we will show that $H_0\cap W\neq\emptyset$ but $m_H(H_0\cap W)=0$, so that $W$ is not Haar regular.

First observe that $(\Delta(c))_{b_0}=c\in c\Z$, so that $\Delta(c)\in H_0$. Suppose for a contradiction that $H_0\cap W=\emptyset$. Then $\Delta(c)\not\in W$, i.e. there is $b\in\cB$ such that $c\in b\Z$. Hence $c\cA\subseteq b\Z$, so that $c\cA=\{b\}$, because $b\in\cB$ and  $\cB$ is primitive. Hence $b=ca_0=b_0$, so that $\cA=\{a_0\}$, a contradiction, as $\cA$ is Behrend.

We turn to the proof of $m_H(H_0\cap W)=0$.
Let $\cH_W^\ell=\{n\in\{0,\dots,s_\ell-1\}: U_{S_\ell}(\Delta(n))\cap H_0\cap W\neq\emptyset\}$. It suffices to show that $\sum_{n\in\cH_W^\ell}m_H(U_{S_\ell}(\Delta(n))\to0$ as $\ell\to\infty$. As all cylinder sets $U_{S_\ell}(\Delta(n))$ have identical Haar measure $s_\ell^{-1}$, this is  equivalent to $\#\cH_W^\ell/s_\ell\to0$ as $\ell\to\infty$. So let $\ell$ be so large that $b_0\in S_\ell$.
Denote $\cA^\ell=\{a\in\cA:ca\mid s_\ell\}$.
As $c\cA\subseteq\cB$, the sequence $(\cA^\ell)_\ell$ is increasing and exhausts the set $\cA$.

If $n\in\cH_W^\ell$, then $n\in c\Z$ and, by Lemma~\ref{lemma:easy}a,  $n\in\cF_{S_\ell}$. Hence $n=cn'\in\cF_{S_\ell}$ for some $n'\in\Z$. Suppose for a contradiction that $n'\in\cM_{\cA^\ell}$, i.e. there are $k\in\Z$ and $a\in\cA_\ell$ such that $n'=ka$. Then $n=kca$, where $ca\in\cB$ and $ca\mid s_\ell$, so that $ca\in S_\ell$, which contradicts $n\in\cF_{S_\ell}$. Hence $n'\in\cF_{\cA^\ell}$ so that $n\in c\cF_{\cA^\ell}=c\,(\Z\setminus\cM_{\cA^\ell})$. As $\cA$ is Behrend, $\bar{d}(\Z\setminus\cM_{\cA^\ell})\to0$ as $\ell\to\infty$. Hence
\begin{equation*}
\#\cH_W^\ell/s_\ell
\leqslant
\#\left(c(\Z\setminus\cM_{\cA^\ell})\cap[0,s_l)\right)/s_\ell
\le
\#\left((\Z\setminus\cM_{\cA^\ell})\cap[0,s_l)\right)/s_\ell
=
d(\Z\setminus\cM_{\cA^\ell})
\to0\ .
\end{equation*}

Suppose now that $\cB$ is taut.
We must show that for any $k\in\N$ and $U\in\cZ_k$
\begin{equation*}
U\cap W=\emptyset\quad\text{or}\quad m_H(U\cap W)>0\ .
\end{equation*}
So fix some $U=U_{S_k}(\Delta(n))$ such that $m_H(U\cap W)=0$. We have to show that $U\cap W=\emptyset$. Observe first that $U_{S_k}(\Delta(m))=U$ if and only if $m\in s_k\Z+n$.
For $\ell>k$ let
\begin{equation*}
\cG_\ell:=(s_k\Z+n)\cap
\left\{m\in\Z:
U_{S_\ell}(\Delta(m))\cap W=\emptyset\right\}
=
(s_k\Z+n)\cap \cM_{S_\ell},
\end{equation*}
where we used Lemma~\ref{lemma:easy}c for the last equality.
Observe that
\begin{equation*}
\cG_\ell
=
\cG_\ell+s_\ell\Z
=
\left(\cG_\ell\cap[0,s_\ell)\right)+s_\ell\Z\ .
\end{equation*}
Hence,
for each $\ell>k$,
\begin{equation*}
\begin{split}
\underline{d}\left((s_k\Z+n)\cap\cM_\cB\right)
&=
\liminf_{t\to\infty}\frac{\# \left(
(s_k\Z+n)\cap \cM_\cB\cap[0,t)
\right)}{t}
\geqslant
\liminf_{t\to\infty}\frac{\# \left(
(s_k\Z+n)\cap \cM_{S_\ell}\cap[0,t)
\right)}{t}\\
&=
\liminf_{t\to\infty}\frac{\# \left(\cG_\ell\cap[0,t)\right)}{t}
=
\frac{\# \left(\cG_\ell\cap[0,s_\ell)\right)}{s_\ell}.
\end{split}
\end{equation*}
As all $U'\in\cZ_\ell$ have identical Haar measure $m_H(U')=s_\ell^{-1}$ and as $m_H(U\setminus W)=m_H(U)$ by assumption, it follows that
\begin{equation*}
\begin{split}
\underline{d}\left((s_k\Z+n)\cap\cM_\cB\right)
&\geqslant
\limsup_{\ell\to\infty}\frac{\# \left(\cG_\ell\cap[0,s_\ell)\right)}{s_\ell}
=
\limsup_{\ell\to\infty}m_H\left(\bigcup_{U'\in\cZ_\ell,U'\subseteq U\setminus W}U'\right)\\
&=
m_H(U\setminus W)
=m_H(U)
=s_k^{-1}
=
d(s_k\Z+n)\ ,
\end{split}
\end{equation*}
so that
\begin{equation*}
d\left((s_k\Z+n)\cap\cM_\cB\right)=d(s_k\Z+n)\ .
\end{equation*}

Let $q=\gcd(s_k,n)$, $a'=s_k/q$ and $r'=n/q$. Then $\gcd(a',r')=1$ and
$q\Z\cap\cM_\cB=q\Z\cap\cM_{q\cdot\cB'(q)}$, in particular $(s_k\Z+n)\cap\cM_\cB=(s_k\Z+n)\cap\cM_{q\cdot\cB'(q)}$. Hence
\begin{equation*}
\begin{split}
d\left((a'\Z+r')\cap\cM_{\cB'(q)}\right)
&=
q\cdot d\left(q\cdot\left((a'\Z+r')\cap\cM_{\cB'(q)}\right)\right)
=
q\cdot d\left((s_k\Z+n)\cap \cM_{q\cdot\cB'(q)}\right)\\
&=
q\cdot d\left((s_k\Z+n)\cap \cM_{\cB}\right)
=
q\cdot d(s_k\Z+n)
=
q\cdot d\left(q(a'\Z+r')\right)\\
&=
d(a'\Z+r')
=1/a'\ .
\end{split}
\end{equation*}
In view of Lemma 1.17 in \cite{hall-book}, this suffices to conclude that
$\cB'(q)$ is Behrend.

On the other hand, as $\cB$ is taut,
$\cB'(q)$ is pre-taut (Lemma~\ref{lemma:q-taut}), so that $1\in\cB'(q)$ or $\cB'(q)$ is not Behrend (Lemma~\ref{lemma:taut-facts}e). Hence $1\in\cB'(q)$.
This implies
$q\in\cM_\cB$, which in turn implies
$U\cap W=U_{S_k}(\Delta(n))\cap W=\emptyset$ (the property to be proved):
Indeed, if $q\in\cM_\cB$, then there is some $b\in\cB$ with $b\mid q$, and as $q\mid s_k$, this implies $b\mid s_k$, so that $b\in S_k$. From $b\mid q\mid n$ we then conclude that $n\in\cM_{S_k}$, and Lemma~\ref{lemma:easy}a
implies $U_{S_k}(\Delta(n))\cap W=\emptyset$.

It remains to show that the implication $(i)$ $\Rightarrow$ $(iii)$ follows from Proposition~\ref{prop:mariusz-stronger}, which will be proved in the next subsection. So let $h\in W$.
By the proposition there exists $n\in\cF_\cB$ such that $\Delta(n)\in U_S(h)$, hence $\Delta(n)\in U_S(h)\cap(\Delta(\Z)\cap W)$. As this holds for all finite $S\subset\cB$, this proves the claim.

\subsection{Tautification of the set $\cB$ and regularization of the window $W$}
In \cite[Section 4.2]{BKKL2015} the authors provide a construction that associates to each (non-taut) set $\cB$ a taut set $\cB'$ such that
$\cF_{\cB'}\subseteq\cF_\cB$ but $\overline{d}(\cF_\cB\setminus\cF_{\cB'})=0$, and such that the two Mirsky measures
$\nu_\eta$ and $\nu_{\eta'}$
determined by $\cB$ and $\cB'$ coincide.
$\cB$ and $\cB'$ determine groups $H$ resp. $H'$ with windows $W$ resp. $W'$, and while the window $W$ is not Haar regular (if $\cB$ is non-taut), the window $W'$ is Haar regular because of Theorem~\ref{theo:taut-regular}.

On the abstract level one can also pass
from the window $W\subseteq H$ to its \emph{Haar regularization} $W_{reg}:=\supp(m_H|_W)$ (introduced in \cite{KR2016}), which also determines the same Mirsky measure on $\{0,1\}^\Z$. However, $W_{reg}$ will not be a window of the particular arithmetic type defined in \eqref{eq:W}, in particular it need not be aperiodic. The construction of $\cB'$ given $\cB$ in \cite{BKKL2015} suggests an obvious factor map $f:H\to H'$, and we expect that also $f(W_{reg})=W'$, so that in this sense the regularization of $W$ and the tautification of $\cB$ are two sides of the same medal.

The following example illustrates this discussion.

\begin{example}\label{ex:ape1}
Let $\cP=\{p_1,p_2,\ldots\}$ denote the set of primes. Let $\cB:=\bigcup_{i\geq1}p_i^2(\cP\setminus\{p_i\})$. Note that $\cB$ is primitive. It is not taut, because it contains rescalings of  Behrend sets. The corresponding taut set is $\cB'=\{p_i^2:\:i\geq1\}$, which generates  the square-free system.~\footnote{Note that $\eta(n)=1$  at all square-free numbers and also at $p_i^k$ for $i\geq1$ and $k\geq2$.}
\end{example}

\subsection{The property $\overline{\Delta(\Z)\cap W}=W$}\label{subsec:property(iii)}

\begin{proof}[Proof of Proposition~\ref{prop:mariusz-stronger}]
Given $h\in W$, we need to show that for each finite $S\subset\cB$
the set
\begin{equation*}
\cL_S(h)
:=
\left\{n\in\cF_\cB:\text{ $h_b=n$ mod~$b$ for each $b\in S$}\right\}
\end{equation*}
has asymptotic density $m_H(U_S(h)\cap W)>0$.

By Theorem~\ref{theo:taut-regular}, the tautness assumption on $\cB$ implies that $W$ is Haar regular, so that indeed
\begin{equation*}
m_H(U_S(h)\cap W)>0\ .
\end{equation*}

Let $\cB=\{b_1,b_2,\ldots\}$ and, for $K\geq 1$,  $W_K:=\{g\in H:\:g_i\neq0\text{ for }i=1,\ldots,K\}$. Then $W_K$ is clopen and $W\subseteq W_K$. Moreover, $W_{K}\supseteq W_{K+1}$ and $\bigcap_KW_K=W$. Fix $\varepsilon>0$. We now choose $K\geq1$ so that
\begin{equation}\label{tcrt1}
m_H(W_K\setminus W)<\varepsilon\ .
\end{equation}
Since $U_S(h)\cap W_K$ is clopen (and $T$ is strictly ergodic)
\begin{equation}\label{tcrt2}
\left|\frac1N\sum_{n\leq N}\1_{U_S(h)\cap W_K}(T^n0)-m_H\left(U_S(h)\cap W_K\right)\right|<\varepsilon
\end{equation}
for all $N\geq N_0$. Moreover, we can choose $N_1$ so that for $N\geq N_1$, we also have
\begin{equation}\label{tcrt3}
\left|\frac1N\sum_{n\leq N}\1_{U_S(h)\cap W_K}(T^n0)-\frac1N\sum_{n\leq N}\1_{U_S(h)\cap W}(T^n0)\right|<\varepsilon\ .
\end{equation}
Indeed, if
\begin{equation*}
T^n0=\Delta(n)\in  \left(U_S(h)\cap W_K\right)\setminus \left(U_S(h)\cap W\right)\subset W_K\setminus W\ ,
\end{equation*}
then (by setting $\cB_K=\{b_1,\ldots,b_K\}$), we have
\begin{equation*}
n\in \cF_{\cB_K}\cap \cM_{\cB}=\cM_{\cB}\setminus\cM_{\cB_K}\ .
\end{equation*}
Therefore, by the Davenport-Erd\"os theorem \cite[Eq.~(0.67)]{hall-book}, we can choose first $K\geq1$ sufficiently large so that  $\overline{d}(\cM_{\cB}\setminus\cM_{\cB_K})<\varepsilon$ and then $N_1$ so that
\begin{equation*}
\frac1N\sum_{n\leq N}\1_{W_K\setminus W}(T^n0)
=
\frac1N\sum_{n\leq N}\1_{\cM_{\cB}\setminus\cM_{\cB_K}}(n)<\varepsilon
\end{equation*}
for all $N\geq N_1$, so in particular \eqref{tcrt3} holds.
In view of \eqref{tcrt1}, \eqref{tcrt2}  and \eqref{tcrt3}, it follows that
\begin{equation*}
\lim_{N\to\infty}\frac1N\sum_{n\leq N}\1_{U_S(h)\cap W}(T^n0)=m_H(U_S(h)\cap W)\ .
\end{equation*}
As $T^n0=\Delta(n)\in U_S(h)\cap W$ if and only if $n\in\cL_S(h)$, this finishes the proof.
\end{proof}

\begin{example}
($\overline{\Delta(\Z)\cap W}=W$ does not imply tautness)\\
\label{example:non(iii)to(i)}
Suppose that $(m_k,r_k)$, ${k\in\N}$, is an enumeration of all coprime pairs of natural numbers. For any $k$ choose a prime $p_k\in r_k+m_k\Z$ such that $p_k>2^{k+1}$. Let $\cB=\cP\setminus\{p_k:k\in\N\}$. Clearly $\cB$ is primitive, and $\cM_{\{p_k:k\in\N\} }$ has upper density less than or equal to $\sum_{k=1}^{\infty}1/2^{k+1}=1/2$. Thus $d(\cM_{\cB})=1$ and $\cB$ is not taut \cite[Corollary 0.14]{hall-book}.
But $\overline{\Delta(\Z)\cap W}=W$. Indeed, let $h=(h_b)_{b\in\cB}\in W$ and take any finite set $S\subset \cB$. We are going to show that $U_S(h)\cap W\cap \Delta(\Z)\neq \emptyset$.  Let $n\in\Z$ be such that $n= h_b$ mod $b$ for $b\in S$. Since $h\in W$, $b$ does not divide $n$ for any $b\in S$, i.e. $\lcm(S)$ and $n$ are coprime. Then $(\lcm(S),n)=(m_k,r_k)$ for some $k$, and the prime number $p_k$ belongs to arithmetic progression $r_k+m_k\Z=n+\lcm(S)\Z$, in other words $\Delta(p_k)\in U_S(\Delta(n))=U_S(h)$.
Finally, $\Delta(p_k)\in W$, because the prime number $p_k$ does not belong to $\cB$ and hence also not to $\cM_\cB$.
\end{example}

\subsection{$X_\eta$ and $X_\varphi$}
\label{subsec:light-tails}
The set $\cB\subseteq\N$ has \emph{light tails}, if
\begin{equation}
\lim_{K\to\infty}\overline{d}\left(\cM_{\{b\in\cB:b>K\}}\right)=0\ .
\end{equation}
If $\cB$ has light tails, then $\cB$ is taut, but the converse doses not hold \cite[Section 4.3]{BKKL2015}. Here we prove:
\begin{proposition}\label{prop:Mariusz4.2} If $\cB$ has light tails, then $X_\eta= X_\varphi$.
\end{proposition}

\begin{proof} Let $H=(h_k)\in H$ and $n\in\N$. We are going to show that $\varphi(h)[-n,n]=\eta[l+1,l+2n+1]$ for some $l\in\Z$. We know that $\varphi(h)(i)=1$ if and only if $h_j+i$ is not a multiple of $b_j$ for any $j\in\N$. For any $i\in[-n,n]$ such that $\varphi(h)(i)=0$ let $k_i$ be such that $b_{k_i}|h_{k_i}+i$.

Let $K\in\N$ be such that the set $\mathscr{A}:=\{b_1,\ldots,b_K\}$ contains $b_{k_i}$, for $i\in[-n,n]$ and any $b_k$ with $k>K$ has a prime factor $p>2n+1$.
Since $h\in H$, there exists $m\in Z$ such that
\begin{equation}\label{d1}
 m=h_k \mod b_k
\end{equation}
for all $k\le K$. It follows that
\begin{equation*}
(\supp \varphi(h)\cap [-n,n])+m=[-n+m,n+m]\cap \cF_{\cA}
\end{equation*}
Indeed, if $i\in \supp \varphi(h)\cap [-n,n]$,  then $h_k+i$ is not a multiple of $b_k$ for any $k\in \N$. By (\ref{d1}) we get that $m+i$ is not a multiple of $b_k$ for any $k\le K$, that is, $m+i\in \cF_{\cA}$. On the other hand, if $i\notin {\rm supp}
 \varphi(h)\cap [-n,n]$, then $b_{k_i}|h_{k_i}+i$. Since $k_i\le K$, again by (\ref{d1}), we obtain $b_{k_i}|m+i$, that is $m+i\notin \cF_{\cA}$.

By \cite[Proposition 5.11]{BKKL2015} \footnote{Assume that $\cB\subset \N$ has light tails and $\cB^{(n)}\subset \cA\subset \cB$.
Suppose that
\begin{equation}\label{dziedz:as1}
\{k+1,\ldots,k+n\}\cap {\cal M}_{\cA}=\{k+i_0,k+i_1,\ldots,k+i_r\}
\end{equation}
for some $1\le i_0,\ldots,i_r\le n$, $r<n$. Then the density of $k'\in\N$ such that
$$
\{k'+1,\ldots,k'+n\}\cap {\cal M}_{\cB}=\{k'+i_0,k'+i_1,\ldots,k'+i_r\}
$$
is positive. (Here $
\cB^{(n)}:=\{b\in \cB : p\le n \text{ for any }p\in{\rm Spec}(b)\}
$. If $\cB$ is primitive, then $\cB^{(n)}$ is finite.)}
there exists $l\in\Z$ such that
\begin{equation*}
([-n+m,n+m]\cap \cF_{\cA})+l+n+1-m=[l+1,l+2n+1]\cap \cF_{\cB}
\end{equation*}
It follows that $\varphi(h)[-n,n]=\eta[l+1,l+2n+1]$.
\end{proof}

We now present a Behrend set (hence a non-taut set), for which $X_\eta$ is a strict subeset of $X_\varphi$.

\begin{example} Let $\cB=\{p_2,p_3,\ldots\}=\{3,5,7,11,\ldots\}$ - the set of all odd prime numbers. Since we are in the coprime case,
$$
H=\prod_{k=2}^\infty\Z/p_k\Z.$$
Now, $\eta=\varphi(\Delta(0))$ is the characteristic function of the ${\cB}$-free set  $\{\pm 2^m:\:m\geq0\}$. We  compute an initial block of $\varphi(h)$ for
$$
h=(0,1,0,0,\ldots)\in H.$$
We have $\varphi(h)(0)=0$,
$\varphi(h)(1)=1$, $\varphi(h)(2)=1$, $\varphi(h)(3)=0$, $\varphi(h)(4)=0$\footnote{If we add 4 to each coordinate of $h$, we obtain the sequence $(1,0,4,4,\ldots)$, whence $\varphi(h)(4)=0$.}, $\varphi(h)(5)=1$, $\varphi(h)(6)=0$, $\varphi(h)(7)=0$ and $\varphi(h)(8)=1$. It follows that the block $11001001$ appears on $\varphi(h)$. But there is no block $\underline{a}$ of length 8 appearing on $\eta$ and such that $11001001\le\underline{a}$. Indeed, the two neighboring 1's at the beginning of $\underline{a}$ could only appear at the positions 1,2 or -2,-1 in $\eta$. In the both cases this would force $\eta(5)=1$, which is not true. This shows that $\varphi(h)\not\in X_\eta$, although it belongs to $X_\varphi$.\footnote{Indeed, $\varphi(h)$ does not even belong to $\widetilde X_\eta$, the hereditary closure of $X_\eta$, see \cite{BKKL2015}.}
\end{example}

\begin{question} If $\cB$ is taut, is then $X_\eta=X_\varphi$?~\footnote{We recall that in case of $\cB$ taut, the Mirsky measure is supported on $X_\eta$.}
\end{question}

\section{Minimality/proximality of $X_\eta$ and topological properties of $W$}\label{sec:proof-theo-minimal}
Throughout this section we assume that $\cB$ is primitive.
\subsection{Arithmetic of $\cB$ and topology of $W$, part II}
\label{subsec:arithmetics-window}
Recall from \eqref{eq:A_S-def} that $\cA_S:=\{\gcd(b,\lcm(S)): b\in\cB\}$
and $\cF_{\cA_S}\subseteq\cF_\cB$
for  $S\subset\cB$. If $S\subseteq S'\subset\cB$, then the following inclusions and implications are obvious:
\begin{equation}\label{eq:inclusions}
S\subseteq S'\subseteq\cA_{S'}\subseteq\cM_{\cA_S}
\Rightarrow\cM_S\subseteq\cM_{S'}\subseteq \cM_{\cA_{S'}}\subseteq\cM_{\cA_S}
\Rightarrow \cF_{\cA_S}\subseteq\cF_{\cA_{S'}}\subseteq\cF_{S'}\subseteq\cF_S\ .
\end{equation}
Let $\cE:=\bigcup_{S\subset\cB}\cF_{\cA_S}$ and observe that $\cE\subseteq\cF_\cB$.

\begin{lemma}\label{lemma:interiorW}
\begin{compactenum}[a)]
\item For all $S\subset\cB$ and $n\in\Z$ we have:
$U_S(\Delta(n))\subseteq W\Leftrightarrow n\in\cF_{\cA_S}$.
\item {If $(S_k)_k$ is a filtration of $\cB$ with finite sets and $\lim_k\Delta(n_{S_k})=h$ (see Remark \ref{limits}), then $h\in \inn(W)$ if and only if $n_{S_k}\in\cF_{\cA_{S_k}}$ for some $k$.}
\item For all $n\in\Z$ we have:
$\Delta(n)\in{\inn(W)}\Leftrightarrow n\in\cE$\ .
\item $\inn(W)=\emptyset\Leftrightarrow \cE=\emptyset
\Leftrightarrow \forall S\subset\cB: \cF_{\cA_S}=\emptyset
\Leftrightarrow \forall S\subset\cB: 1\in\cA_S$
\end{compactenum}
\end{lemma}
\begin{proof}
\begin{compactenum}[a)]
\item
As $U_S(\Delta(n))$ is clopen,
\begin{equation*}
\begin{split}
U_S(\Delta(n))\not\subseteq W
&\Leftrightarrow
\exists m\in n+\lcm(S)\cdot\Z\ \exists c\in\cB: c\mid m\\
&\Leftrightarrow
\exists c\in\cB\
\exists k\in\Z: c\mid n+k\cdot\lcm(S)\\
&\Rightarrow
\exists c\in\cB:\ \gcd(c,\lcm(S))\mid n\\
&\Leftrightarrow
\exists k\in\cA_S: k\mid n\\
&\Leftrightarrow
n\not\in\cF_{\cA_S}\ .
\end{split}
\end{equation*}
{That the only implication is also an equivalence is a consequence of the CRT. Indeed, if $\gcd(c,\lcm(S))\mid n$,
then there exist $k,l\in\Z$ such that $l\cdot c-k\cdot\lcm(S)=n$, thus $c\mid n+k\cdot\lcm(S)$. }
\item {Assume that $h\in \inn(W)$, that is $U_S(h)\subseteq W$ for some $S$. Then, for $k$ such that $S\subseteq S_k$, we have
$U_{S_k}(\Delta(n_{S_k}))=U_{S_k}(h)\subseteq W$, which is equivalent to $n_{S_k}\in\cF_{\cA_{S_k}}$ by a). Conversely, if   $n_{S_k}\in\cF_{\cA_{S_k}}$ then, again by a), $U_{S_k}(\Delta(n_{S_k}))=U_{S_k}(h)\subseteq W$ and $h\in\inn(W)$. }
\item Follows from a).
\item Follows from c).
\end{compactenum}
\end{proof}
Recall from \eqref{def:Cinf} that
$\Ainf:=\{n\in\N: \forall_{S\subset\cB}\ \exists_{S': S\subseteq S'}: n\in\cA_{S'}\setminus S'\}$.
\begin{lemma}\label{lemma:ASk-C}
\begin{compactenum}[a)]
\item If $(S_k)_k$ is a filtration of $\cB$ with finite sets, then
\begin{equation*}
\limsup_{k\to\infty}\cA_{S_k}\setminus S_k=\Ainf\ .
\end{equation*}
\item For each $n\in\Ainf$ there is
a filtration $(S_k)_k$ of $\cB$ with finite sets such that
\begin{equation*}
n\in\bigcap_{k\in\N}\cA_{S_k}\setminus S_k\ .
\end{equation*}
\end{compactenum}
\end{lemma}

\begin{proof}
a) Assume that $n\in\cA_{S_k}\setminus S_k$ for infinitely many $k$, and let $S\subset\cB$.
Then there is $k$ such that $S\subseteq S_k$ and $n\in\cA_{S_k}\setminus S_k$. Hence $n\in\Ainf$.
Conversely, let $n\in\cC_{\infty}$. There is a finite set $S_1$ such that $n\in\cA_{S_1}\setminus S_1$. Assume that we have constructed sets $S_1\subset S_2\subset \ldots \subset S_k$ with the property that $n\in\cA_{S_i}\setminus S_i$ for $i=1,\ldots,k$ and
$\{1,\ldots,k\}\cap\cB\subset S_k$. Then there is a set $S_{k+1}$ containing $S_{k}\cup\{k+1\}$ and such that $n\in\cA_{S_{k+1}}\setminus S_{k+1}$. In this inductive way we construct a filtration $(S_k)_k$ as required.
\\
b) follows from a).
\end{proof}

\begin{lemma}\label{lemma:E=FBC}
The sets $\cE$ and $\Ainf$ are related by the identity
$$
\cE=\cF_{\cB\cup{\Ainf}}=\cF_\cB\cap\cF_{\Ainf}.
$$
\end{lemma}

\begin{proof}
Let $n\in\cE$ and chose $S$ such that $n\in\cF_{\cA_S}$. Take arbitrary $b\in\cB$ and $c\in \Ainf$. There exists a finite set $S'$ such that
$S\cup\{b\}\subseteq S'$ and $c\in\cA_{S'}\setminus S'$. Since $\cF_{\cA_S}\subseteq \cF_{\cA_{S'}}$, $n\in \cF_{\cA_{S'}}$, hence neither $b$ nor $c$ divides $n$.
We have proved that $\cE\subseteq \cF_{\cB\cup\Ainf}$.
In order to prove the other inclusion  assume that $n\in \N$ and that for any $S$ there exists $c_S\in\cA_S$ dividing $n$.
As $n$ has only finitely many divisors, it has a divisor $c$ such that there exists a filtration $(S_k)_k$ of $\cB$ such that $c\in\cA_{S_k}$ for any $k\in \N$. If $c\notin\cB$, then $c\in\cA_{S_k}\setminus S_k$ for any $k\in \N$. This proves $n\notin\cF_{\cB\cup\Ainf}$.
\end{proof}

\begin{lemma}\label{lemma:E=FB}
$\Ainf=\emptyset$ if and only if $\cE=\cF_\cB$.
\end{lemma}
\begin{proof}
If $\Ainf=\emptyset$, then $\cE=\cF_\cB$ by Lemma~\ref{lemma:E=FBC}. Conversely, assume that
$\cE=\cF_\cB$. Then $\cF_\cB\subseteq\cF_{\Ainf}$ by Lemma~\ref{lemma:E=FBC}, so that
$\Ainf\subseteq\cM_{\Ainf}\subseteq\cM_\cB$. Suppose for a contradiction that there exists some $n\in\Ainf$. Then there is $b\in\cB$ such that $n\in b\Z$, i.e. $b\mid n$, and there is
a finite set $S=S_k\subset\cB$ such that
$n\in\cA_{S}\setminus S$, see Lemma~\ref{lemma:ASk-C}b. Hence there exists $b'\in\cB$ such that $n=\gcd(\lcm(S),b')$. It follows that $b\mid n\mid b'$, which is impossible, because $\cB$ is assumed to be primitive.
\end{proof}

\begin{proposition}\label{prop:filtation}
The following conditions are equivalent:
\begin{compactenum}[(i)]
\item $W\neq\overline{\inn(W)}$
\item For any filtration $S_0\subset S_1\subset\ldots \subset \cB$ of $\cB$ with finite subsets $S_k$, there exists a number $d$ such that $d\in \cA_{S_k}\setminus S_k$, for infinitely many $k\in\N$.
\item There exists a filtration $S_0\subset S_1\subset\ldots \subset \cB$ of $\cB$ with finite subsets $S_k$ and there exists a number $d$ such that $d\in \cA_{S_k}\setminus S_k$, for every $k\in\N$.
\item There are $d\in\N$ and an infinite pairwise
coprime set $\cA\subseteq\N\setminus\{1\}$ such that $d\,\cA\subseteq\cB$.
\end{compactenum}
\end{proposition}

\begin{proof}
$(i)\Rightarrow (ii)$: Let $h=(h_b)\in W\setminus \overline{\inn(W)}$. There exists  $S$ such that $U_S(h)\cap \inn(W)=\emptyset$. We can assume that any $b\in \cB$ such that $b|\lcm(S)$, belongs to $S$.\footnote{Otherwise we can incorporate all such $b$'s into $S$, there are finitely many of them.} Let $n$ be a number such that
\begin{equation}\label{f1}
n= h_b \mod b
\end{equation}
for $b\in S$.

Then $\Delta(n+k\lcm(S))\in U_S(h)$, hence $\Delta(n+k\lcm(S))\notin \inn(W)$ for any $k\in\Z$. This means (see Lemma \ref{lemma:interiorW}) that for any finite set $T$, in particular for any $T=S_k$, the arithmetic progression $n+\lcm(S)\Z$ is contained in $\cM_{\cA_T}$. Since the set $\cA_T$ is finite, it follows that $\cA_T$ contains a divisor of $\gcd(n,\lcm(S))$\footnote{Apply Dirichlet theorem on primes in arithmetic progressions.}. There is only finitely many divisors of $\gcd(n,\lcm(S))$, hence one of them, denote it by $d$, appears in $\cA_{S_k}$ for infinitely many $k$. To finish the proof it is enough to observe that $d\notin\cB$ (consequently, $d\notin S_k$, for any $k$). Indeed, otherwise $d\in S$, by our assumption on $S$. Moreover, $d|n$  and then, by (\ref{f1}), $d|h_b$, where $b=d$, which leads to a contradiction with the assumption $h\in W$.

$(ii)\Rightarrow (iii)$: obvious

$(iii)\Rightarrow (i)$: Assume that $d\in\cA_{S_k}\setminus S_k$ for any $k$. Then $d\notin \cM_{\cB}$\footnote{Otherwise $d$ is divisible by some $b\in\cB$. On the other hand, $d$ divides some $b'\in\cB$ as a member of $\cA_{S_k}$, which in view of the fact that $\cB$ is primitive, leads to the conclusion that $d=b=b'\in\cB$. But it is not true, since $d\notin S_k$ for any $k$.}, hence $\Delta(d)\in W$. We prove that $\Delta(d)\notin\overline{\inn(W)}$. It is enough to show that $U_{S_0}(\Delta(d))\cap \inn(W)=\emptyset$.

Assume that $h=(h_b)\in U_{S_0}(\Delta(d))\cap \inn(W)$. It means that
\begin{equation}\label{f2}
d=h_b \mod b\; \mbox{\rm for any}\;b\in S_0\ ,
\end{equation}
 and there exists a finite set $T\subset\cB$ such that $U_T(h)\subset W$. We can assume that $T=S_k$ for some $k$.

Let $m\in\Z$ be such that
\begin{equation}\label{f3}
m=h_b \mod b\; \mbox{\rm for any}\;b\in S_k\ .
\end{equation}
   Let $c\in \cB$ be such that $\gcd(c,\lcm(S_k))=d$. {Clearly, $c\notin S_k$, since $d\notin \cB$.} Since $U_{S_k}(h)\subset W$, it follows
   that there exists $b\in S_k$ such that
\begin{equation}\label{f4}
\gcd(c,b)\; \mbox{\rm  does not divide}\; h_b \ ,
\end{equation}
{Indeed, otherwise there would exist $l\in\Z$ such that $l\equiv h_b$ mod $b$ for $b\in S_k$ and $l= 0$ mod $c$, hence $\Delta(l)\in U_{S_k}(h)$, but $\Delta(l)\notin W$, a contradiction.}
Thus, in view of (\ref{f3}),
\begin{equation}\label{f4a}
\gcd(c,b) \; \mbox{\rm  does not divide}\; m\ .
\end{equation}

On the other hand,
\begin{equation}\label{f5}
\gcd(c,b)|\gcd(c,\lcm(S_k))=d\ .
\end{equation}
Since $d\in\cA_{S_0}\setminus S_0$ we get
\begin{equation}\label{f6}
d|\lcm(S_0)\ .
\end{equation}
By (\ref{f2}) and (\ref{f3}),
\begin{equation}\label{f7}
\lcm(S_0)|m-d\ .
\end{equation}
Now, (\ref{f5}), (\ref{f6}) and (\ref{f7}) imply $\gcd(c,b)|m$, a contradiction with (\ref{f4a}).

$(iii)\Rightarrow (iv)$: Assume that $d\in\cA_{S_k}\setminus S_k$ for any $k$. Then
\begin{equation}\label{eq:search-for-coprime}
\forall k\in\N\ \exists b_k\in\cB\setminus S_k:\ d=\gcd(b_k,\lcm(S_k))\ .
\end{equation}
As $d\not\in S_k$, we have $b_k\neq d$ for all $k$.
We choose a subsequence $b_{k_1},b_{k_2},\dots$ of $(b_k)_k$ in the following way: Let $k_1=1$, and given $k_1,\dots,k_j$, let
\begin{equation*}
k_{j+1}:=\min\left\{k\in\N: b_{k_1},\dots,b_{k_j}\in S_{k_{j+1}}\right\}.
\end{equation*}

Let $a_j=b_{k_j}/d$ for all $j\in\N$ and denote $\cA=\{a_j:j\in\N\}$.
Then $\cA\subseteq\N$  and $d\,\cA\subseteq \cB$ by construction.
Suppose that $1\in\cA$. Then $d\in\cB$, a contradiction to
\eqref{eq:search-for-coprime}, as $\cB$ is primitive. Hence $\cA\subseteq\N\setminus\{1\}$.

It remains to prove that
$\cA$ is pairwise coprime.
Suppose for a contradiction that there is a prime number $p$ dividing some $a_i$ and $a_j$, $i<j$. Then $pd\mid b_{k_i}$ and $pd\mid b_{k_j}$.
As $b_{k_i}\in S_{k_j}$, it follows that $pd\mid\lcm(S_{k_j})$, so that $pd\mid\gcd(b_{k_j},\lcm(S_{k_j}))=d$ (see \eqref{eq:search-for-coprime}), which is impossible.

$(iv)\Rightarrow (iii)$: Let $d\in\N$ and $\cA=\{a_1<a_2<\dots\}$ be as in $(iv)$. Then $d\not\in\cB$, because $\cB$ is primitive.
For $k\in\N$ let $S_k=\cB\cap\{1,\dots,k\}\cup\{da_k\}$. As all $a_j$ are pairwise coprime, there are $j_1<j_2<\dots\in\N$ such that $a_{j_k}$ is coprime to $\lcm(S_k)$. On the other hand, $d\mid\lcm(S_k)$. Hence $d=\gcd(da_{j_k},\lcm(S_k))\in\cA_{S_k}$. As $d\not\in\cB$, we see that $d\in\cA_{S_k}\setminus S_k$ for all $k\in\N$.
\end{proof}

\begin{proposition}\label{prop:Staszek-prop-equiv}
The following conditions are equivalent:
\begin{compactenum}[(i)]
\item $W$ is topologically regular, i.e. $W=\overline{\inn(W)}$.
\item There are no $d\in\N$ and no infinite pairwise
coprime set $\cA\subseteq\N\setminus\{1\}$ such that $d\,\cA\subseteq\cB$.
\item $\Ainf=\emptyset$.
\item $\cE=\cF_\cB$.
\end{compactenum}
\end{proposition}

\begin{proof}
The equivalence of $(i)$ and $(ii)$ follows from Proposition~\ref{prop:filtation},
that of $(iii)$ and $(iv)$ from Lemma~\ref{lemma:E=FB}..
In view of Lemma~\ref{lemma:ASk-C},
Proposition~\ref{prop:filtation} finally implies the equivalence of $(i)$ and $(iii)$, too.
\end{proof}

\begin{lemma}\label{lemma:DeltaZ-1}
$\Delta(\Z)\cap\left(\overline{\inn(W)}\setminus\inn(W)\right)=\emptyset$.
\end{lemma}

\begin{proof}
Assume $\Delta(m)\in \overline{\inn(W)}\setminus\inn(W)$. {Then for any $S\subset \cB$ there exists $n_S\in\Z$ such that $\Delta(n_S)\in U_S(\Delta(m))\cap \inn(W)$.} It means that for any $S$ there exist: a finite set $T_S\subset \cB$, (we can assume that $S\subset T_S$), $b_S\in\cB$ and $n_S\in\Z$ such that (see Lemma \ref{lemma:interiorW} c)):
 \begin{compactenum}[$\bullet$]
 \item $\lcm(S)|m-n_s$ (that is, $\Delta(n_s)\in U_S(\Delta(m))$)
 \item $\gcd(b_S,\lcm(T_S))$ does not divide $n_S$ ($\Delta(n_S)$ is chosen to be an element of $U_S(\Delta(m))\cap \inn(W)$)
 \item $\gcd(b_S,\lcm(T_S))|m$ (since $\Delta(m)\notin \inn(W)$)
  \end{compactenum}
  Then  $\gcd(b_S,\lcm(T_S))$ does not divide $\lcm(S)$.
 Let us iterate: $S_0$ is arbitrary and $S_{k+1}:=T_{S_k}$, $c_k:=\lcm(S_{k+1})$, $d_k:=\gcd(b_{S_k},\lcm(S_{k+1}))$.
 We have:
 \begin{compactenum}[$\bullet$]
 \item $c_k|c_{k+1}$
 \item $d_k|m$
 \item $d_k|c_k$
 \item $d_k$ does not divide  $c_{k-1}$
 \end{compactenum}
{ Since $d_k|m$ for every $k$, the sequence $(\lcm(d_1,\ldots d_k))_k$ stabilizes on $\lcm(d_1,\ldots d_{k_0})$ for some $k_0$, which means  $d_l$ divides $\lcm(d_1,\ldots d_{k_0})$, and consequently $d_l$ divides $\lcm(c_1,\ldots,c_{k_0})=c_{k_0}$, for any $l$, a contradiction.}
\end{proof}

For $x\in\{0,1\}^\Z$ denote $\supp x:=\{n\in\Z: x(n)=1\}$. Following \cite{BKKL2015} we consider the set
\begin{equation*}
\begin{split}
Y
:=&
\left\{x\in\{0,1\}^\Z: |\supp x\text{ mod }b|=b-1\text{ for all }b\in\cB\right\}\\
=&
\left\{x\in\{0,1\}^\Z: \text{for all }b\in\cB\text{ there is exactly one }r\in\{0,\dots,b-1\}\text{ with }\supp x\cap(b\Z+r)=\emptyset\right\}.
\end{split}
\end{equation*}
As  $\supp\eta=\cF_\cB$ is disjoint from $b\Z$ for all $b\in\cB$, we have
\begin{equation}\label{eq:eta-in-Y}
\eta\in Y
\Leftrightarrow
\forall b\in\cB\ \forall r\in\{1,\dots,b-1\}: \cF_\cB\cap (b\Z+r)\neq\emptyset\ .
\end{equation}

\begin{lemma}\label{lemma:Y-eta}
If $\overline{\Delta(\Z)\cap W}=W$, then $\eta\in Y$.
\end{lemma}

\begin{proof}
For $b\in\cB$ and $r\in\{0,\dots,b-1\}$ let $V_b(r):=\{h\in H: h_b=r\}$ and observe that these sets are open and closed in $H$.
Hence $\overline{\Delta(\Z)\cap V_b(r)\cap W}=V_b(r)\cap W$, because $\overline{\Delta(\Z)\cap W}=W$.

Suppose for a contradiction that $\eta\not\in Y$. Then (\ref{eq:eta-in-Y}) implies
that there are $b\in\cB$ and $r\in\{1,\dots,b-1\}$ such that
\begin{equation*}
\Delta(\Z)\cap V_b(r)\cap W=\emptyset\ ,
\end{equation*}
which implies that also $V_b(r)\cap W=\emptyset$. Hence
\begin{equation*}
V_b(r)\subseteq
W^c=
\bigcup_{b'\in\cB}V_{b'}(0)\ ,
\end{equation*}
and as $V_b(r)$ is compact and the $V_{b'}(0)$ are open, there is a finite $S\subset\cB$ such that
\begin{equation*}
V_b(r)\subseteq
\bigcup_{b'\in S}V_{b'}(0)\ ,
\end{equation*}
In other words, whenever $h_b=r$ for some $h\in H$, then $h_{b'}=0$ for some $b'\in S$.
Applied to any $h=\Delta(n)$ this yields:
\begin{equation*}
n\in b\Z+r\;\Rightarrow\;n\in\bigcup_{b'\in S}b'\Z\ .
\end{equation*}
{Since $r$ is not divisible by $b$, we can assume that $b\notin S$.}
Let $q=\gcd(b,r)$, $\tilde{b}=b/q$, $\tilde{r}=r/q$. Then $q\,(\tilde{b}\Z+\tilde{r})=b\Z+r\subseteq\cM_S$, so that $\cM_{\tilde{b}\Z+\tilde{r}}\subseteq\cM_{S'(q)}$ by Lemma~\ref{lemma:introductory}.  But $d(\cM_{\tilde{b}\Z+\tilde{r}})=1$ by Dirichlet's theorem, whereas $d(\cM_{S'(q)})<1$, because $S'(q)\subseteq\{1,\dots,\max S\}$ is finite {and $1\notin S'(q)$} \footnote{{As $q|b$ and $\cB$ is primitive,  $q\notin S$, thus $1\notin S'(q)$.}}. This is a contradiction.
\end{proof}

\begin{remark}
Together with Theorem~\ref{theo:taut-regular} this shows that
$\eta\in Y$ whenever $\cB$ is taut. This implication was proved previously in \cite[Corollary 4.27]{BKKL2015}.
\end{remark}

Lemma~\ref{lemma:Y-eta} provides the implication
\begin{equation*}
\overline{\Delta(\Z)\cap W}=W\Rightarrow \eta\in Y.
\end{equation*}
The reverse implication does not hold, as is shown by the next example.

\begin{example} Observe that for every $k\in\Z$ there exists a prime divisor $p_k$ of $5+12k$ such that
\begin{equation}\label{ex_1}
p_k\neq 1\mod 12 \;\;\text{and}\;\;p_k\neq -1\mod 12
\end{equation}
Let $$\cB=\{4,6\}\cup\{p_k:k\in\Z\}$$
Let us enumerate the elements of $\cB$ as $b_0,b_1,b_2,\ldots$ and $b_0=4, b_1=6$.
Observe that
\begin{equation}\label{ex_2}
5+12\Z\subset\cM_{\cB}
\end{equation}
Since niether 2 nor 3 divides an element of the progression $5+12\Z$, in view of (\ref{ex_1}) we see that $1,2,3,11,22\in\cF_{\cB}$. It follows that
\begin{equation}\label{ex_3}
|\supp\cF_{\cB} \mod 4|=3\;\;\text{and} \;\; |\supp\cF_{\cB} \mod 6|=5
\end{equation}

We claim that
\begin{equation}\label{ex_4}
|\supp\cF_{\cB} \mod b_k|=b_k-1\;\text{for any} \; k\ge 2
\end{equation}
It is clear that $\gcd(12,b_k)=1$ for any $k\ge 2$. Let $k\ge 2$ and take arbitrary $r\in\{1,\ldots,b_k-1\}$. There exists $r'\in\Z$ such that
\begin{equation}\label{ex_5}
\left\{\begin{array}{c}
r'\equiv r\mod b_k\\
r'\equiv 1 \mod 12
\end{array}\right.
\end{equation}
Then $\gcd(12b_k,r')=1$ and, by Dirichlet Theorem, there exists a prime number $q$ of the form $q=12b_kl+r'$ for some $l\in\Z$. Since, by (\ref{ex_5}), $q\equiv 1\mod 12$, $q\in\cF_{\cB}$ by (\ref{ex_1}). Moreover, $q\equiv r\mod b_k$ by (\ref{ex_5}).

Thus the claim (\ref{ex_4}) follows. Clearly, (\ref{ex_4}) and (\ref{ex_3}) yield $\eta\in Y$.

We shall construct $h\in W$ such that $h\notin\overline{\Delta(\Z)\cap W}$. We denote $S_k=\{b_0,b_1,\ldots b_k\}$. Inductively we construct a sequence $(n_{S_k})$ of integers satisfying:
\begin{compactenum}[a)]
\item $n_{S_1}=5$
\item $\lcm(S_k)|n_{S_{k+1}}-n_{S_k}$ for $k=1,2,\ldots$
\item $n_{S_k}\in\cF_{S_k}$ for $k=1,2,\ldots$
\end{compactenum}
Assume that $n_{S_1},\ldots,n_{S_k}$ have been constructed. If $b_{k+1}$ does not divide $n_{S_k}$, we set $n_{S_{k+1}}=n_{S_k}$. Otherwise we set
$n_{S_{k+1}}=n_{S_k}+\lcm(S_k)$. The conditions a), b), c) follow easily by induction.
Let
$$
h=\lim_k\Delta(n_{S_k})
$$
Thanks to c), $h\in W$.

But
$$
U_{S_1}(h)\cap\Delta(\Z)\cap W=U_{S_1}(\Delta(5))\cap\Delta(\Z)\cap W=\Delta(5+\12\Z)\cap W=\emptyset
$$
the last equality by (\ref{ex_2}). (Clearly, $d(\cM_{\cB})=1$ and $\cB$ is not taut.)
\end{example}

\subsection{Proof of Theorem~\ref{theo:minimality}}\label{subsec:proof-theo-minimality}

\begin{lemma}\label{lemma:Aurelia}
\footnote{The authors are indebted to A.~Bartnicka for pointing out and proving this lemma.}
If $\cB$ is primitive and $\eta$ is a Toeplitz sequence, then $\cB$ is taut.
\end{lemma}

\begin{proof}
Suppose that $\cB$ is not taut. Then there are $c\in\N$ and a Behrend set $\cA$ such that  $c\cA\subseteq \cB$. Hence
\begin{equation}\label{eq:Aurelia}
d(\cM_\cB\cap c\Z)=c^{-1}\ ,
\end{equation}
because $\cM_\cA$ has density one. As $\cB$ is primitive, $c$ must be $\cB$-free.
So $\eta(c)=1$, and (since $\eta$ is Toeplitz) there exists $m\in\N$ such that $c+m\Z\subseteq \cF_\cB$.
But then
\begin{equation*}
\underline{d}(\cF_\cB\cap c\Z)
\geqslant
\underline{d}((c+m\Z)\cap c\Z)
=
d(\lcm(c,m)\Z)
=\lcm(c,m)^{-1}>0,
\end{equation*}
which contradics \eqref{eq:Aurelia}.
\end{proof}

\begin{lemma}\label{lemma:almost-periodic-Y}
Assume that $\eta\in Y$. If $\eta=\1_{\cF_\cB}$ is almost periodic (i.e. if the orbit closure of $\eta$ is minimal), then
$X_\eta\subseteq Y$.
\end{lemma}
\begin{proof}
Fix $k\geq1$. Since $\eta\in Y$,  the support of $\eta$ taken mod $b_k$ misses exactly one residue class mod $b_k$ (that is, it misses zero). Let $B$ be a block on $\eta$ such that its support mod $b_k$ misses exactly one residue class mod $b_k$. Since $\eta$ is almost periodic, the block $B$ appears on $\eta$ with bounded gaps. It follows that if $C$ is any sufficiently long block that appears on $\eta$, its support misses exactly one residue class. Clearly this property passes to limits in the product topology, so each $y=\lim S^{m_i}\eta$ is also in $Y$.
\end{proof}

In general, we can define a map $\theta:Y\to\prod_{k\geq1}\Z/b_k\Z$ by setting
\begin{equation*}
\theta(y)=g=(g_k)_{k\geq1}\text{ iff } \supp y\cap(b_k\Z-g_k)=\emptyset\text{ for all }k\geq1.
\end{equation*}
Remark 2.51 in \cite{BKKL2015} tells us that
\begin{equation*}
\theta(Y\cap X_\eta)\subset H\ ,
\end{equation*}
while Remark 2.52 says that $\theta$ is continuous.

\begin{corollary}\label{coro:theta}
By the definitions of $\varphi$ and $\theta$, we have $\theta\circ\varphi(h)=h$ provided $\varphi(h)\in Y$. In particular, $\theta(\eta)=0$ and $\theta$ is continuous at $\eta$. Moreover, $\theta$ is equivariant.
\end{corollary}

For any map $\psi:X\to Y$ denote by $C_\psi\subseteq X$ the set of continuity points of this map.
\begin{lemma}\label{lemma:minimality}
Let $(X,S)$ and $(Y,T)$ be compact dynamical systems and assume that $(X,S)$ is minimal. Let $\psi:X\to Y$ be a map satisfying $\psi\circ S=T\circ\psi$. Then $\overline{\psi(C_\psi)}$ is a minimal subset of $Y$.
\end{lemma}
\begin{proof}
Denote by $Z:=\overline{\{(x,\psi(x)): x\in X\}}$ the closure of the graph of $\psi$ and note that a fibre $Z_x=\{(x,y):y\in Z\}$ is a singleton, if and only if $x\in C_\psi$. Let $Z_0:=\overline{\{(x,\psi(x)): x\in C_\psi\}}$. We claim that $Z_0\subseteq A$ whenever $A$ is a non-empty closed $S\times T$-invariant subset of $Z$. Indeed, $\pi_X(A)$ is a non-empty closed $S$-invariant subset of $X$, so $\pi_X(A)=X$ by minimality of $(X,S)$. In particular, $C_\psi\subseteq\pi_X(A)$. As all $A_x\subseteq Z_x$ with $x\in C_\psi$ are singletons, $\{(x,\psi(x)):x\in C_\psi\}\subseteq A$. Hence also $Z_0\subseteq A$.

This shows that $Z_0$ is a minimal subset of $X\times Y$ (and, by the way, that it is the only minimal subset of $Z$). It follows that $\pi_Y(Z_0)$ is a minimal subset of $Y$, and so it remains to show that $\psi(C_\psi)\subseteq \pi_Y(Z_0)$. But, for $x\in C_\psi$, $(x,\psi(x))\in Z_0$, and so $\psi(x)\in\pi_Y(Z_0)$.
\end{proof}

Denote by $C_\varphi$ the set of all points in $H$ at which $\varphi:H\to\{0,1\}^\Z$ is continuous.
\begin{lemma}\label{lemma:C_varphi}
\begin{compactenum}[a)]
\item $C_\varphi=\left\{h\in H:\ (h+\Delta(\Z))\cap\partial W=\emptyset\right\}$.
\item $C_\varphi+\Delta(1)=C_\varphi$.
\item $\overline{\varphi(C_\varphi)}$ is the unique minimal subset $M$.
\end{compactenum}
\end{lemma}
\begin{proof}
a) This is proved by direct inspection, see e.g. \cite[Lemma 6.1]{KR2015}.\\
b) This is obvious.\\
c) This follows from Lemma~\ref{lemma:minimality}.
\end{proof}
\begin{proof}[Proof of Theorem~\ref{theo:minimality}]
We start with a list of implications, which, when suitably combined, prove the assertions a) - e) of Theorem~\ref{theo:minimality}.
Most of these implications can be proved without assuming that $\cB$ is primitive and that $\overline{\Delta(\Z)\cap W}=W$.
Therefore we indicate explicitly, for which implications we use these extra assumptions.
\\[2mm]
\emph{Proof of the equivalence of B1 -- B4:}
These equivalences follow from Proposition~\ref{prop:Staszek-prop-equiv}.
\\[2mm]
\emph{Proof of B1 $\Rightarrow$ B6:}
Observe first that $0\in H$ belongs to $C_\varphi$ if and only if $\Delta(\Z)\cap\partial W=\emptyset$, see Lemma~\ref{lemma:C_varphi}. But $\Delta(\Z)\cap\partial W=\Delta(\Z)\cap\left(\overline{\inn(W)}\setminus\inn(W)\right)$ in view of B1, and this intersection is empty by Lemma~\ref{lemma:DeltaZ-1}. As $\inn(W)\neq\emptyset$ and as $H=\overline{\Delta(\Z)}$, $\Delta(\Z)\cap W\neq\emptyset$ and hence $\varphi(0)\neq(\dots,0,0,0,\dots)$.
\\[2mm]
\emph{Proof of B6 $\Rightarrow$ B5:}
Let $\cB=\{b_1,b_2,\ldots\}$ and assume (B6) that $0\in C_\varphi$,
i.e. $\Delta(\Z)\cap\partial W=\emptyset$, and $\eta\neq(\dots,0,0,0,\dots)$.
Now, take $n\in\Z$. Either $n\in\cM_{\cB}$ - then $\eta(n)=0$, so $b_s\mid n$ for some $s\geq1$ and $\eta(n+jb_s)=0$ for each $j\in\Z$. Or
$n\in \cF_{\cB}$, i.e. $\Delta(n)\in W$. As $\Delta(\Z)\cap\partial W=\emptyset$ by assumption, this implies $\Delta(n)\in\inn(W)$, so that $n\in\cE=\bigcup_{S\subset\cB}\cF_{\cA_S}$ by Lemma~\ref{lemma:interiorW}. Hence there is a finite subset $S\subset\cB$ such that $n\in\cF_{\cA_S}$. As $\lcm(\cA_S)=\lcm(S)$, this implies
\begin{equation*}
n+\lcm(S)\,\Z\subseteq\cF_{\cA_S}\subseteq\cE\subseteq\cF_\cB\ .
\end{equation*}
Hence $\eta(n+j\lcm(S))=1$ for each $j\in\Z$. This proves that $\eta$ is a Toeplitz sequence
different from $(\dots,0,0,0,\dots)$.
\\[2mm]
\emph{Proof of B5 $\Rightarrow$ B1:}
Assume that $\eta$ is a Toeplitz sequence. Then $\cB$ is taut by Lemma~\ref{lemma:Aurelia}, hence $\overline{\Delta(\Z)\cap W}=W$ by Theorem~\ref{theo:taut-regular}. Now B1 follows from the chain of the next three implications.
\\[2mm]
\emph{Proof of B5 $\Rightarrow$ B8:}
Each Toeplitz sequence is almost periodic \cite{Do}, \cite[Theorem 4]{JK1969},   i.e. its orbit closure is minimal.
\\[2mm]
\emph{Proof of B8 $\Rightarrow$ B7:}
If $X_\eta=M$, then $\eta\in M$, and $\eta\neq(\dots,0,0,0,\dots)$, because otherwise the minimality of $X_\eta$ implies $X_\eta=\{(\dots,0,0,0,\dots)\}$, contradicting $\card(X_\eta)>1$.
\\[2mm]
\emph{Proof of B7 $\Rightarrow$ B1  (assuming that $\overline{\Delta(\Z)\cap W}=W$):}
Assume that $(\dots,0,0,0,\dots)\neq\eta\in M=\overline{\varphi(C_\varphi)}$.
Then $M=X_\eta\subseteq Y$ by Lemma~\ref{lemma:almost-periodic-Y}, and
there is a sequence $h_1,h_2,\dots\in C_\varphi$ such that $\eta=\lim_{i\to\infty}\varphi(h_i)$. Consider $n\in\Z$ with
$\Delta(n)\in W$, i.e. such that $\eta(n)=1$. In particular $\eta=\varphi(0)\neq(\dots,0,0,0,\dots)$. Corollary~\ref{coro:theta} implies
$\lim_{i\to\infty}h_i=\lim_{i\to\infty}\theta(\varphi(h_i))=\theta(\eta)=0$. Then
$1=\eta(n)=\lim_{i\to\infty}\varphi(h_i)(n)$, i.e. $h_i+\Delta(n)\in W$ for all sufficiently large $i$. As $h_i\in C_\varphi$, we have $h_i+\Delta(\Z)\cap\partial W=\emptyset$ (Lemma~\ref{lemma:C_varphi}). Hence
$h_i+\Delta(n)\in \inn(W)$ for all sufficiently large $i$, what implies
that $\Delta(n)=\lim_{i\to\infty}h_i+\Delta(n)\in\overline{\inn(W)}$.
This proves that $\Delta(\Z)\cap W\subseteq\overline{\inn(W)}$.
Hence
$W=\overline{\Delta(\Z)\cap W}\subseteq\overline{\inn(W)}$, i.e. $W$ is topologically regular.
\\[2mm]
\emph{Proof of B7 $\Rightarrow$ B8:}
As $\eta\in M$, also $X_\eta\subseteq M$, and hence $X_\eta= M$. As $\eta\neq(\dots,0,0,0,\dots)$, $X_\eta$ contains no fixed point. Hence $\card(X_\eta)>1$.
\\[2mm]
\emph{Proof of B1 $\Rightarrow$ B9  (assuming that $\cB$ is primitive):}
The window $W$ is aperiodic because of Proposition~\ref{p:ape3}, and it is topologically regular by B1. As B1 $\Rightarrow$ B8, $X_\eta$ is minimal. Therefore Corollary 1a) of \cite{KR2015}, together with Lemmas 4.5 and 4.6 of the same reference, implies B9.
\\[2mm]
\emph{Proof of B9 $\Rightarrow$ B8:}
This is trivial.
\end{proof}

\begin{proposition}\label{eta_in_Y}
Assume that the window $W$ is topologically regular. Then $X_\eta\subseteq Y$.
\end{proposition}
\begin{proof}
We start proving that $\eta\in Y$. Assume
the contrary, that is, there are $b_0\in\cB$ and $r\in\{1,\ldots,b_0-1\}$ such that
\begin{equation}\label{eta_in_Y1}
r+b_0\Z\subset\cM_{\cB}.
\end{equation}
 Let $a=\gcd(r,b_0)$ and $r'=r/a$, $b_0'=b_0/a$.
 (\ref{eta_in_Y1}) yields that for any $k\in \N$ there exists $b_k\in\cB$ such that
 \begin{equation*}
 b_k\mid a(r'+kb_0').
 \end{equation*}
  Let $J=\{k\in\N:r'+kb_0'\;\text{is prime}\}$. By Dirichlet Theorem the set $J$ is infinite. As $\cB$ is primitive, $b_k$ does not divide $a=\gcd(r,b_0)$. Hence
 $$
 b_k=\gcd(a,b_k)\,(r'+kb_0')
 $$
 for any $k\in J$.
 Since $a$ has only finitely many divisors, there exists a divisor $a'$ such that
 $$
 b_k=a'(r'+kb_0')
 $$
for infinitely many $k\in J$. Thus we obtain a contradiction with the condition (B4) of Theorem~\ref{theo:minimality}, which is equivalent to (B1) $W=\overline{\inn W}$. Thus $\eta\in Y$.

Assume now that $x\in X_\eta$ and let $b\in\cB$.
As $\eta\in Y$, there is $N_b\in\N$ such that $\card\left(\supp(\eta|_{[0:N_b]})\mod b\right)=b-1$. As $X_\eta$ is minimal by (B8) of Theorem~\ref{theo:minimality}, there is
$n\in\N$ such that $\supp(x|_{[n:n+N_b]})=\supp(\eta|_{[0:N_b]})$. Hence
\begin{displaymath}
\card\left(\supp(x)\mod b\right)\geqslant \card\left(\supp(x|_{[n:n+N_b]})\mod b\right)
\geqslant\card\left(\supp(\eta|_{[0:N_b]})\mod b\right)=
b-1\ ,
\end{displaymath}
so that $x\in Y$, because $\card\left(\supp(x)\mod b\right)\leqslant b-1$ for all $x\in X_\eta$, see Footnote~\ref{foot:b-1} to Remark~\ref{rem:Y}.
\end{proof}

\subsection{Proof of Theorem~\ref{theo:proximal}}\label{subsec:proof-theo-proximal}
The equivalence of C1, C2 and C3 follows from
Lemma~\ref{lemma:interiorW}.
If C1 holds, i.e. if $\inn(W)=\emptyset$, then $\varphi(C_\varphi)=\{(\dots,0,0,0,\dots)\}$ is a shift invariant set \cite[Proposition 3.3d with Remark 3.2b]{KR2015}, so that $M=\overline{\varphi(C_\varphi)}=\{(\dots,0,0,0,\dots)\}$. This is C6, and Theorem 3.8 in \cite{BKKL2015} shows that C4, C5, C6 and C7 are all equivalent.

We finish by proving C5 $\Rightarrow$ C3: Consider any finite $S\subset \cB$. As $\cB\subseteq\cM_{\cA_S}$ by definition of the set $\cA_S$, C5 implies that $1\in\cA_S$.

\section{The sequence $\cB$ and Haar measure}\label{sec:B-Haar}
\subsection{Measure and density}\label{subsec:measure-density}

\begin{lemma}
$m_H(W)=1-\underline d(\cM_\cB)=\bar d(\cF_\cB)$.
\end{lemma}

\begin{proof}
For $S\subset\cB$ denote by
$\cU_S$
the family of all sets $U_S(\Delta(n))$ that are contained in $W^c$ and by $\cup\cU_S$ the union of these sets. Then
\begin{equation*}
m_H(W^c)
=
\sup_S m_H(\cup\cU_S)
=
\sup_S
\frac{\#\cU_S}{\lcm(S)}
\geqslant
\sup_S \frac{\#(\cM_S\cap\{1,\dots,\lcm(S)\})}{\lcm(S)}
=
\sup_S d(\cM_S)
=\underline d(\cM_\cB)
\end{equation*}
by Lemma~\ref{lemma:easy}a,
and similarly
\begin{equation*}
m_H(W^c)
\leqslant
\sup_S \frac{\#(\cM_{\cB\cap\cA_S}\cap\{1,\dots,\lcm(S)\})}{\lcm(S)}
=
\sup_S d(\cM_{\cB\cap\cA_S})
\leqslant
\underline d(\cM_\cB)
\end{equation*}
by Lemma~\ref{lemma:easy}b.
\end{proof}

\begin{corollary}\textbf{\cite[Theorem 4.1]{BKKL2015}}
$\cB$ is  a Besicovich sequence
if and only if $\cF_\cB$ is generic for the Mirsky measure.
(As $n\in\cF_\cB$ iff $\Delta(n)\in W$, it would be more precise to say that the sequence $(\Delta(n))_n$ is generic for the Mirsky measure.)
\end{corollary}

\begin{proof}
If $\cB$ is Besicovich, then $d(\cF_\cB)=m_H(W)$, so that $\cF_\cB$ has maximal density. Hence it is generic for the Mirsky measure, see \cite[Theorem 5b]{KR2015}. On the other hand, if $\cF_\cB$ is generic for (any) measure, then its frequency of ones converges in particular, which means that its asymptotic density exists.
\end{proof}

\begin{lemma}
$m_H(\inn(W))
=
\sup_S d(\cF_{\cA_S})
\leqslant
\underline d(\cE)$
\end{lemma}

\begin{proof}
For $S\subset\cB$ denote by
$\cU^o_S$
the family of all sets $U_S(\Delta(n))$ that are contained in $\inn(W)$ and by $\cup\cU^o_S$ the union of these sets.
Recall from Lemma~\ref{lemma:interiorW}a that
$\#\cU^o_S=\#(\cF_{\cA_S}\cap\{1,\dots,\lcm(S)\})$.
 Then
\begin{equation*}
m_H(\inn(W))
=
\sup_S m_H(\cup\cU^o_S)
=
\sup_S
\frac{\#\cU^o_S}{\lcm(S)}
=
\sup_S \frac{\#(\cF_{\cA_S}\cap\{1,\dots,\lcm(S)\})}{\lcm(S)}
=
\sup_S d(\cF_{\cA_S})
\end{equation*}
\end{proof}

\begin{lemma}\label{lemma:Haar-boundary}
$m_H(\partial W)=\inf_S\bar d(\cM_{\cA_S}\setminus\cM_\cB)
\leqslant \inf_S d(\cM_{\cA_S\setminus\cB})$.
\end{lemma}

\begin{proof}
\begin{equation*}
\begin{split}
m_H(\partial W)
&=
m_H(W)-m_H(\inn(W))
=
\bar d(\cF_\cB)-\sup_S d(\cF_{\cA_S})
=\inf_S\left(\bar d(\cF_\cB)-d(\cF_{\cA_S})\right)\\
&=
\inf_S\bar d(\cM_{\cA_S}\setminus\cM_\cB)
\end{split}
\end{equation*}
\end{proof}

\subsection{Regular Toeplitz sequences}\label{subsec:Toeplitz}

Let $\cB=\{b_1,b_2,\ldots\}$. For each $k\geq1$, consider the sequence
$$
b_1,\ldots,b_k,c^{(k)}_{k+1},c^{(k)}_{k+2},\ldots,$$
where
$$
c^{(k)}_{k+i}:={\rm gcd}({\rm lcm}(b_1,\ldots,b_k),b_{k+i}),\;i\geq1.$$
Then:
\begin{equation*}\label{cirm1}
c^{(k)}_i|{\rm lcm}(b_1,\ldots,b_k), \mbox{whence $\{c^{(k)}_{k+i}:\:i\geq1\}$ is finite}.\end{equation*}
\begin{equation*}\label{cirm2}
\mathscr{A}_{\{b_1,\ldots,b_k\}}=\{b_1,\ldots,b_k\}\cup\{c^{(k)}_{k+i}:\:i\geq1\}.\end{equation*}
\begin{equation*}\label{cirm3}
c^{(k)}_{k+1}|b_{k+1}.\end{equation*}
\begin{equation*}\label{cirm4}
c^{(k)}_{k+1+i}|c^{(k+1)}_{k+1+i},\;\text{for each }i\geq1.
\end{equation*}
Moreover, following Lemma~\ref{lemma:easy}c, there is an increasing sequence $(k_n)$ such that
\begin{equation}\label{cirm0}
\cB\cap \mathscr{A}_{\{b_1,\ldots,b_{k_n}\}}=\{b_1,\ldots,b_{k_n}\}.
\end{equation}

We assume that $W\subset H$ is topologically regular, so by Remark~\ref{remark:window-toeplitz},  $\eta=\1_{\cF_{\cB}}$ is a Toeplitz sequence.  We set $s_k:={\rm lcm}(b_1,\ldots,b_k)$ and would like now to examine the sequence $(s_k)$ as a periodic structure of $\eta$. More precisely, we would like to see for how many $n\in[1,s_k]$, we have $\eta(n)=\eta(n+js_k)$ for each $j\in\Z$. We call any such $n$ to be ``good''. Now, if $n\in \cF_{\mathscr{A}_{\{b_1,\ldots,b_k\}}}$, then
$n+s_k\Z\subset \cF_{\mathscr{A}_{\{b_1,\ldots,b_k\}}}$, so $n$ is good. Otherwise, $n\in \cM_{\mathscr{A}_{\{b_1,\ldots,b_k\}}}$. Then either
$n\in \cM_{b_1,\ldots,b_k}$ and then clearly $\eta(n+js_k)=0$ for each $j\in\Z$, so again $n$ is good, or
$$n\in \cM_{\{c^{(k)}_{K+i}:\:i\geq1\}}\setminus\cM_{\{b_1,\ldots,b_k\}}.$$
Only for such $n$, we are not sure that $n$ is good. Moreover, note that in view of~\eqref{cirm4}, we have
\begin{equation*}\label{cirm5}
\cM_{\{c^{(k+1)}_{k+1+i}:\:i\geq1\}}\subset \cM_{\{c^{(k)}_{k+i}:\:i\geq1\}},\end{equation*}
so the sequence $(d(\cM_{\{c^{(k)}_{k+i}:\:i\geq1\}}))_k$ is decreasing, and so is the sequence $(\overline{d}(\cM_{\{c^{(k)}_{k+i}:\:i\geq1\}}\setminus\cM_{\cB}))$.
Therefore, by taking into account~\eqref{cirm0}, the infimum of this sequence is equal to the liminf, in fact to the limit and  we have
\begin{equation}\label{cirm6}
\inf_{S\subset \cB}\overline{d}(\cM_{\mathscr{A}_S})\setminus \cM_{\cB})=\liminf_{k\to\infty}\overline{d}(\cM_{\{c^{(k)}_{k+i}:\:i\geq1\}\setminus\cM_{\{b_1,\ldots,b_k\}}})\ .
\end{equation}
\begin{definition}
Let
$\eta=\1_{\cF_{\cB}}$ be a Toeplitz sequence.
It is a \emph{regular Toeplitz sequence} for the periodic structure $(s_k)$, $s_k={\rm lcm}(b_1,\ldots,b_k)$, if the $\liminf$ in \eqref{cirm6} is zero.
\end{definition}
\noindent
Now, using Lemma~\ref{lemma:Haar-boundary}, the identity in \eqref{cirm6} shows the following result.
\begin{proposition}\label{prop:toeplitz} If $W$ is topologically regular, then $\eta=\1_{\cF_{\cB}}$ is a regular Toeplitz sequence for the periodic structure $(s_k)$, $s_k={\rm lcm}(b_1,\ldots,b_k)$, if and only if $m_H(\partial W)=0$.
\end{proposition}

\begin{example} Assume that $\{b'_k:\:k\geq1\}$ is a coprime set of odd numbers and let $b_k=2^kb'_k$. Then $c^{(k)}_{k+i}=2^k$ for each $i\geq1$.
Hence, we have even $d(\cM_{\{c^{(k)}_{k+i}:\:i\geq1\}})\to 0$,
in particular $\eta$ is a regular Toeplitz sequence for the periodic structure $(s_k)$ with
$s_k=2^kb_1'\cdots b_k'$.
This example comes from \cite{BKKL2015}.
\end{example}

We will now show that we can obtain Toeplitz sequences also in case $m_H(\partial W)>0$.
\begin{example}
We will construct $\cB=\{b_1,b_2,\ldots\}$ such that this set is thin (hence taut) and
that
$\lim_{k\to\infty}\inf(\{c^{(k)}_{k+i}:\:i\geq1\}\setminus\{b_1,\ldots,b_k\})=\infty$, which, by Proposition~\ref{prop:Staszek-prop-equiv}, implies that $W$ is topologically regular (and hence $\eta$ is Toeplitz by Remark~\ref{remark:window-toeplitz}). Let $\delta_k>0$ and $\sum_{k\geq1}\delta_k<1/16$.

We start with $b_1=2^3$ and set $c^{(1)}_{1+i}=2$ for each $i\geq1$. Suppose that
a sequence
$$
b_1,\ldots, b_k, c^{(k)}_{k+1},c^{(k)}_{k+2},\ldots$$
has been defined. We require that this sequence satisfies:
\begin{equation*}\label{cirm7} c^{(k)}_{k+i}|{\rm lcm}(b_1,\ldots,b_k),\;i\geq1,\end{equation*}
\begin{equation*}\label{cirm8} c^{(k)}_{k+i}\notin \{b_1,\ldots,b_k\}, \;i\geq1,\end{equation*}
\begin{equation*}\label{cirm9} \mbox{for each $i\geq1$, $|\{j\geq1:\:c^{(k)}_{k+j}=c^{(k)}_{k+i}\}|=+\infty$}.\end{equation*}
We will now show how to define $c^{(k+1)}_{k+2}$, $c^{(k+1)}_{k+3}$, $\ldots$ and then $b_{k+1}$.
Recall an elementary lemma.

\begin{lemma}\label{l:cirm1} Let $F_1, F_2$ be finite sets of natural numbers such that ${\rm gcd}(f_1,f_2)=1$ for each $f_i\in F_i$, $i=1,2$.
Then $d(\cM_{F_1\cdot F_2})=d(\cM_{F_1}\cap\cM_{F_2})=d(\cM_{F_1})d(\cM_{F_2})$.
\end{lemma}

Choose $P\subset \mathscr{P}\setminus{\rm spec}\,\{b_1,\ldots,b_k\}$, so that (by Lemma~\ref{l:cirm1})
\begin{equation*}\label{cirm10}
d(\cM_{P\cdot\{c^{(k)}_{k+i}:\:i\geq1\}})\geq d(\cM_{\{c^{(k)}_{k+i}:\:i\geq1\}})-\delta_k.\end{equation*}
In view of \eqref{cirm9}, $$
\{i\geq2:\:c^{(k)}_{k+i}=c^{(k)}_{k+1}\}=\{r_1,r_2,\ldots\}$$
Let $P=\{q_1,\ldots,q_t\}$. Then set
$$
c^{(k+1)}_{k+r_1+tj}:=c^{(k)}_{k+1}q_1,c^{(k+1)}_{k+r_2+tj}:=c^{(k)}_{k+1}q_2,\ldots, c^{(k+1)}_{k+r_t+tj}:=c^{(k)}_{k+1}q_t$$
for each $j=0,1,\ldots$
If $2\notin \{i\geq2:\:c^{(k)}_{k+i}=c^{(k)}_{k+1}\}$ then repeat the same construction with the set
$\{i\geq2:\: c^{(k)}_{k+i}=c^{(k)}_{k+2}\}$. Since (by~\eqref{cirm7}) the set $\{c^{(k)}_{k+i}:\:i\geq1\}$ is finite, our construction of the sequence $(c^{(k+1)}_{k+1+i})_i$ is done in finitely many steps. Finally, we set $b_{k+1}:=c^{(k)}_{k+1}\prod_{q\in P}q$ (or, if needed,   $b_{k+1}:=c^{(k)}_{k+1}\prod_{q\in P}q^{\alpha_{k+1}}$ for any $\alpha_{k+1}\in\N$). Note that
$$
{\rm gcd}({\rm lcm}(b_1,\ldots,b_k),b_{k+1})=c^{(k)}_{k+1}$$
since $P\cap\{b_1,\ldots,b_k\}=\emptyset$. More than that, by the construction, we also have
$$
{\rm gcd}({\rm lcm}(b_1,\ldots,b_k),b_{k+i})=c^{(k)}_{k+i}\text{ for each }i\geq1.$$
Moreover, it is not hard to see that the new sequence
$$
b_1,\ldots,b_k,b_{k+1},c^{(k+1)}_{k+2},c^{(k+1)}_{k+3},\ldots$$
satisfies \eqref{cirm7}-\eqref{cirm9}. Furthermore, $\cB=\{b_1,b_2,\ldots\}$ satisfies the other requirements mentioned at the beginning of the construction so that $\eta$ is a Toeplitz sequence and $W$ is topologically regular. Note that in our construction
$d(\cM_{\{c^{(1)}_{1+i}:\:i\geq1\}})=1/2$. Moreover, by \eqref{cirm10}
$$
d(\cM_{\{c^{(k)}_{k+i}:\:i\geq1\}})\geq d(\cM_{\{c^{(1)}_{1+i}:\:i\geq1\}})-\sum_{j=1}^k\delta_k\geq \frac14$$
for each $k\geq1$. Finally notice that $d(\cM_{\cB})\leq \sum_{k\geq1}1/b_k$, which (by construction) can be made smaller than 1/8.
It follows that $\lim_{k\to\infty}\overline{d}(\cM_{\mathscr{A}_{\{b_1,\ldots,b_k\}}}\setminus \cM_{\cB})>0$, whence $m_H(\partial W)>0$.
\end{example}

\section{The maximal equicontinuous factor of $X_\eta$}\label{sec:mef}

\subsection{The period groups of $W$ and of $\inn(W)$}\label{subsec:period-group}

Given a subset $A\subseteq H$, denote by
\begin{equation*}
H_A:=\left\{h\in H: W+h=W\right\}
\end{equation*}
the \emph{period group} of $A$. The set $A\subseteq H$ is \emph{topologically aperiodic}, if $H_A=\{0\}$.
The following simple observations are proved in \cite[Lemma 6.1]{KR2016}:
\begin{itemize}
\item $H_A\subseteq H_{\bar A}\cap H_{\inn(A)}$.
\item If $A$ is closed, then $H_{\inn(A)}=H_{\overline{\inn(A)}}$ is closed.
\end{itemize}

\begin{proposition}\label{p:ape3}
Assume that $\cB$ is primitive. Then the window $W$ is topologically aperiodic.
\end{proposition}
\begin{proof} Suppose that $h=(h_b)_{b\in\cB}\neq0$ and
\begin{equation}\label{ape4}
W+h=W\ .
\end{equation}
Since $h\neq0$, there is $b\in\cB$ such that $b$ does not divide $h_b$. Let $n:=\gcd(b,h_b)$.
Then $n\in\cF_{\cB}$, as otherwise there exists $b'\in \cB$ such that $b'\mid n$; but then $b'\mid b$ a contradiction ($\cB$ is assumed to be primitive). Hence $\Delta(n)\in W$. There are $x,y\in\Z$ such that $n=xh_b+yb$, whence $b\mid n-xh_b$. It follows that
$\Delta(n)-xh\notin W$, a contradiction with~\eqref{ape4}.
\end{proof}

If $W$ is topologically regular, then clearly $\inn(W)$ is topologically aperiodic, as well. Otherwise $H_{\inn(W)}$ may be non-trivial, as we will see in the course of this section.

Recall from \eqref{eq:Bprim} that for any set $A\subseteq\N$,
\begin{equation*}
\cM_A = \cM_{\prim A}
\end{equation*}
If $\prim A$ is finite, then $\cM_A$ is a union of finitely many arithmetic progressions. Let $c_A$ denote the period of $\cM_A$, that is, the least natural number such that $c_A+\cM_A=\cM_A$.

\begin{lemma}\label{period-of-MA} Assume that $A,B\subset\N$ are finite
\begin{compactenum}[a)]
\item If $c+\cM_A=\cM_A$ for some $c\in\N$, then $\lcm(\prim A)\mid c$.
\item $c_A=\lcm(\prim A)$
\item if $A\subset\cM_B$ then $c_B\mid c_A$
\end{compactenum}
\end{lemma}

\begin{proof}
\begin{compactenum}[a)]
\item For any $a\in \prim A$ we have $a+c\Z\subseteq\cM_{\prim A}$, from which it follows that
there exists $a'\in \prim A$ such that $a'\mid \gcd(a,c)$. But, as $\prim A$ is primitive, that means that $a'=a$ and $a\mid c$. We conclude that $\lcm(\prim A)\mid c$.
\item Clearly $\cM_A+\lcm(\prim A)=\cM_A$, thus $c_A\mid \lcm(\prim A)$ and the assertion follows from a).
\item If $A\subset\cM_B$ then $\prim A\subset \cM_{\prim B}$, hence $\lcm(\prim B)\mid \lcm(\prim A)$, and we finish by b).
\end{compactenum}
\end{proof}

\begin{lemma}\label{A-goes-back}
Assume that $S\subseteq S'$ are finite subsets of $\cB$, then $\cA_S=\{\gcd(a,\lcm(S)):a\in\cA_{S'}\}$.
\end{lemma}

\begin{proof}
Since $\lcm(S)\mid \lcm(S')$, $\gcd(b,\lcm(S))=\gcd(\gcd(b,\lcm(S')),\lcm(S))$  for any $b\in\cB$ and the assertion follows.
\end{proof}

Let $S_1\subset S_2\subset\ldots\subset S_k\subset\ldots$ be a filtration of $\cB$ with finite sets and denote
\begin{equation*}
s_k:=\lcm(S_k),\; c_k:=c_{\cA_{S_k}}
\end{equation*}
By Lemma \ref{period-of-MA} c) we have $c_l\mid c_{l+1}$ for any $l$. It follows that, for any $k$, the sequence $(\gcd(s_k,c_l))_{l\ge 1}$ stabilizes on a
divisor $d_k$ of $s_k$.
Clearly, since $c_k\mid s_k$,
\begin{equation}\label{nowy_1G}
c_k\mid d_k\mid s_k\ .
\end{equation}
Observe that
\begin{equation}\label{nowy_2G}
d_k=\gcd(s_k,d_{k+1})\ .
\end{equation}
Indeed, there is $l_0\in\N$ such that $d_{k+1}=\gcd(s_{k+1},c_l)$ for all $l>l_0$. Since $s_k\mid s_{k+1}$,
we get
$$
\gcd(s_k,c_l)=\gcd(s_k,\gcd(s_{k+1},c_l))\ .
$$
It follows that $\gcd(s_k,c_l)=\gcd(s_k,d_{k+1})$ for $l>l_0$, and (\ref{nowy_2G}) follows.

By applying (\ref{nowy_2G}) we prove by induction that
\begin{equation}\label{nowy_3G}
d_k=\gcd(s_k,d_{k+j})\ .
\end{equation}
for $j\ge 0$.

\begin{lemma}\label{lemma:equivalent}
Let $(n_k)_{k\in\N}$ be a sequence of integers.
The following are equivalent:
\begin{equation}\label{eq:equiv1S}
\forall k\in\N:\ c_{k}\mid n_{k}\;\text{ and }\; s_k\mid n_{k+1}-n_k\ ,
\end{equation}
and
\begin{equation}\label{eq:equiv2}
\forall k\in\N:\ d_{k}\mid n_{k}\;\text{ and }\; s_k\mid n_{k+1}-n_k\ .
\end{equation}
\end{lemma}
\begin{proof}
As $c_{k}\mid d_{k}$, \eqref{eq:equiv2} implies \eqref{eq:equiv1S}. Conversely, assume that \eqref{eq:equiv1S} holds. We show inductively that for all $j\geqslant0$
\begin{equation}\label{eq:inductive}
\forall k\in\N:\ \gcd(s_{k},c_{k+j})\mid n_{k}\ ,
\end{equation}
and this implies \eqref{eq:equiv2} immediately.

For $j=0$, \eqref{eq:inductive} follows from \eqref{eq:equiv1S}, because $c_{k}\mid s_{k}$.
So suppose that \eqref{eq:inductive} holds for some $j\geqslant0$. Then
\begin{displaymath}
\begin{split}
n_{k+1}&=0\;\mod \gcd(s_{k+1},c_{k+1+j})\;\text{and}\\
n_{k+1}&=n_{k}\mod s_k\ .
\end{split}
\end{displaymath}
Hence $n_k=0\mod \gcd(s_k,s_{k+1},c_{k+1+j})=\gcd(s_k,c_{k+j+1})$ for all $k\in\N$,
i.e. \eqref{eq:inductive} for $j+1$.
\end{proof}

Recall that  $H_{\inn(W)}=\{h\in H:\inn(W)+h=\inn(W)\}$ denotes the period group of $\inn(W)$.

\begin{proposition}\label{prop:periods}
\begin{compactenum}[a)]
\item $h\in H_{\inn(W)}$ if and only if $h=\lim_k\Delta(n_k)$ for some sequence
$(n_k)_k$ satisfying
\begin{equation}\label{eq:n_k}
\forall k\in\N:\ d_k\mid n_k \;\text{ and }\; s_k\mid n_{k+1}-n_k\ .
\end{equation}
Moreover, sequences $(n_k)_k$ can be defined inductively: For $n_1$ there are $s_1/d_1$ choices and, given $n_1,\dots,n_k$, there are precisely $s_{k+1}/\lcm(s_k,d_{k+1})$ many choices for $n_{k+1}$.
\item $H_{\inn(W)}=\{0\}$ if and only if $s_k=d_k$ for all $k\in\N$.
\end{compactenum}
\end{proposition}

\begin{remark}\label{remark:factors}
Observe that, in view of \eqref{nowy_2G},
\begin{equation*}
\frac{s_k}{d_k}\cdot \frac{s_{k+1}}{\lcm(s_k,d_{k+1})}
=
\frac{s_k\, s_{k+1}\gcd(s_k,d_{k+1})}{d_k\,s_k\,d_{k+1}}
=
\frac{s_k\, s_{k+1}\,d_k}{d_k\,s_k\,d_{k+1}}
=
\frac{s_{k+1}}{d_{k+1}},
\end{equation*}
so that
\begin{equation*}
\frac{s_k}{d_k}\mid \frac{s_{k+1}}{d_{k+1}}
\end{equation*}
\begin{equation*}
\frac{s_k}{d_k}
=
\frac{s_1}{d_1}\cdot\prod_{j=1}^{k-1}\frac{s_{j+1}}{\lcm(s_j,d_{j+1})}.
\end{equation*}
\end{remark}

\begin{proof}[Proof of Proposition~\ref{prop:periods}]
a) For each $S_k$ denote by $W_k:=\bigcup_{n\in\cF_{\cA_{S_k}}}U_{S_k}(\Delta(n))$. Then $\inn(W)$ is the increasing union of the sets $W_k$, see Lemma~\ref{lemma:interiorW},
and $U_{S_k}(\Delta(n))\subseteq W_k$ if and only if
$U_{S_k}(\Delta(n))\subseteq\inn(W)$.
Let $h=\lim_k\Delta(n_k)$, where $n_k$ stands for $n_{S_k}$, which was defined in Lemma~\ref{lemma:interiorW}b. Then
\begin{equation}
\forall k\in\N:\  s_k\mid n_{k+1}-n_k\ ,
\end{equation}
and
$h\in H_{\inn(W)}$, if and only if
\begin{equation}\label{eq:W_k-Delta_n_k}
\forall k\in\N:\ \cF_{\cA_{S_k}}+n_k=\cF_{\cA_{S_k}}\ .
\end{equation}
Indeed, let $k\in\N$, $m\in\cF_{\cA_{S_k}}$, and let $g=(g_b)_{b\in\cB}$ be any element from $U_{S_k}(\Delta(m))\subseteq \inn(W)$. Then $g_b=m\mod b$ for all $b\in S_k$. Assume now that $h\in H_{\inn(W)}$. Then $g+h\in\inn(W)$ and
$(g+h)_b=m+n_k\mod b$ for all $b\in S_k$, so that $g+h\in U_{S_k}(\Delta(m+n_k))$.
Hence $U_{S_k}(\Delta(m))+h\subseteq U_{S_k}(\Delta(m+n_k))=U_{S_k}(\Delta(m))+\Delta(n_k)$. In particular, $U_{S_k}(\Delta(m))$ and $U_{S_k}(\Delta(m+n_k))$ have identical Haar measure, and so do $U_{S_k}(\Delta(m))+h$ and $U_{S_k}(\Delta(m+n_k))$. As both are open sets and one is contained in the other, they must coincide.
Hence $U_{S_k}(\Delta(m+n_k))=U_{S_k}(\Delta(m))+h\subseteq\inn(W)+h=\inn(W)$, so that $m+n_k\in\cF_{\cA_{S_k}}$. This proves that $\cF_{\cA_{S_k}}+n_k\subseteq\cF_{\cA_{S_k}}$. As $\cA_{S_k}$ is a finite set, this implies $\cF_{\cA_{S_k}}+n_k=\cF_{\cA_{S_k}}$.

Conversely, assume that \eqref{eq:W_k-Delta_n_k}
 holds, and let $U_{S_k}(\Delta(m))\subseteq\inn(W)$.
Recall that this implies  $U_{S_k}(\Delta(m))\subseteq W_k$, i.e. $m\in\cF_{\cA_{S_k}}$.
Hence, by assumption, also $m+n_k\in\cF_{\cA_{S_k}}$, so that $U_{S_k}(\Delta(m+n_k))\subseteq W_k\subseteq\inn(W)$.
Let $g\in U_{S_k}(\Delta(m))$. Then $g_b=m\mod b$ for all $b\in S_k$, so that
$(g+h)_b=m+n_k\mod b$ for all $b\in S_k$, i.e. $g+h\in U_{S_k}(\Delta(m+n_k))$.
Hence $U_{S_k}(\Delta(m))+h\subseteq U_{S_k}(\Delta(m+n_k))\subseteq\inn(W)$.
As this argument applies to all $k$ and all $U_{S_k}(\Delta(m))\subseteq\inn(W)$,
it proves that $\inn(W)+h\subseteq\inn(W)$.
The same Haar measure argument as before, applied to the open set $\inn(W)$, shows that $\inn(W)+h=\inn(W)$, i.e. $h\in H_{\inn(W)}$.

Condition \eqref{eq:W_k-Delta_n_k} is equivalent to
\begin{equation}
\forall k\in\N:\ c_k=\lcm(\prim\cA_{S_k})\mid n_k\ .
\end{equation}
Invoking Lemma~\ref{lemma:equivalent}, we conclude
\begin{equation}
h\in H_{\inn(W)}
\quad\Leftrightarrow\quad
\forall k\in\N:\ d_{k}\mid n_{k}\;\text{ and }\; s_k\mid n_{k+1}-n_k\ .
\end{equation}
This proves the claimed equivalence.

Now we describe all sequences $(n_k)_{k\in\N}$ which satisfy\eqref{eq:n_k} and $n_k\in\{0,\dots,s_k-1\}$ for all $k$. Denote $q_k:=s_k/d_k$.
\begin{compactenum}[aaaaaa]
\item[$n_1$:] Let $n_1=m_1d_1$ for any $m_1\in\{0,\dots,q_1-1\}$.
\item[$n_2$:] $n_2$ must be chosen such that $n_2=0\mod d_2$  and $n_2=n_1\mod s_1$.
As $\gcd(s_1,d_2)=d_1\mid n_1$ in view of \eqref{nowy_2G}, the CRT guarantees the existence of at least one solution $n_2$, and if $n_2$ is one particular solution, then the set of all solutions is precisely $n_2+\lcm(s_1,d_2)\cdot \Z$. As $n_2$ is to be chosen in $\{0,\dots,s_2-1\}$, there are exactly $s_2/\lcm(s_1,d_2)$ possible choices for $n_2$.
\item[$\vdots$\quad]
\item[$n_{k+1}$:]
$n_{k+1}$ must be chosen such that $n_{k+1}=0\mod d_{k+1}$  and $n_{k+1}=n_k\mod s_k$.
As $\gcd(s_k,d_{k+1})=d_k\mid n_k$ in view of \eqref{nowy_2G}, the CRT guarantees the existence of at least one solution $n_{k+1}$, and if $n_{k+1}$ is one particular solution, then the set of all solutions is precisely $n_{k+1}+\lcm(s_k,d_{k+1})\cdot \Z$. As $n_{k+1}$ is to be chosen in $\{0,\dots,s_{k+1}-1\}$, there are exactly $s_{k+1}/\lcm(s_k,d_{k+1})$ possible choices for $n_{k+1}$.
\end{compactenum}

b) $H_{\inn(W)}=\{0\}$ $\Leftrightarrow$ there is unique choice of the numbers $n_k$ described in a) $\Leftrightarrow$ $s_1/d_1=1$ and $s_{k+1}/\lcm(s_k,d_{k+1})=1$ for any $k$ $\Leftrightarrow$   $d_{k}=s_{k}$ for any $k$, the last equivalence by Remark \ref{remark:factors}.
\end{proof}

\subsection{Proof of Theorem~\ref{theo:MEF}}

\begin{remark}\label{limits}
If $(S_k)_k$ is a filtration of $\cB$ by finite sets and if $h=(h_b)_{b\in\cB}\in H$, then we write $
\lim_k\Delta(n_{S_k})=h$, whenever $n_{S_k}\in\Z$ are numbers such that for every $k\in \N$:
$$
h_b= n_{S_k}\mod b\;\text{\rm for all}\;b\in S_k
$$
Let us denote  $s_k=\lcm(S_k)$. There is an inverse system of groups
\begin{equation*}
\ldots \Z/s_{k+1}\Z\rightarrow \Z/s_{k}\Z \rightarrow \ldots\rightarrow \Z/s_1\Z
\end{equation*}
The homomorphisms are the canonical projections. Observe that $s_k|n_{S_{k+1}}-n_{S_k}$ for any $k$ and the sequence $(n_{S_k}+s_k\Z)_k$ is an element of the inverse limit
$\lim\limits_{\leftarrow}\Z/{s_k}\Z$. In this way we obtain an isomorphism of topological groups
\begin{equation}\label{sigma}
\sigma:\lim\limits_{\leftarrow}\Z/{s_k}\Z\cong H
\end{equation}
given by $(n_{S_k}+s_k\Z)_k\mapsto \lim_k\Delta(n_{S_k})$. Compare Remark 2.32 \cite{BKKL2015}. In particular, the inverse limit does not depend on the filtration $(S_k)_k$.\footnote{The last statement follows from a general property of inverse limits: the inverse limits of cofinal inverse systems are isomorphic, \cite[Chapter II, Section 12]{Fuchs}.}
\end{remark}

\begin{proof}[Proof of Proposition~\ref{prop:period}]
Let $\beta_k:\Z/s_k\Z\rightarrow \Z/d_k\Z$ be the map given by $n+s_{k}\Z\mapsto n+d_{k}\Z$, let $M_k$ be the kernel of $\beta_k$ and let $\alpha_k:M_k\rightarrow \Z/s_k\Z$ be the canonical embedding.
There is a commutative diagram of abelian groups
$$
\begin{array}{ccccccccc}
&&0&&0&&&&0\\
&&\mapdown{}&&\mapdown{}&&&&\mapdown{}\vspace{1ex}\\
\ldots&\mapr{f'_{k+1}}{}&M_{k}&\mapr{f'_{k}}{}&M_{k-1}&\mapr{f'_{k-1}}{}&\ldots&\mapr{f'_{2}}{}&M_{1}\vspace{1ex}\\
&&\mapdown{\alpha_{k}}&&\mapdown{\alpha_{k-1}}&&&&\mapdown{\alpha_{1}}\vspace{1ex}\\
\ldots&\mapr{f_{k+1}}{}&\Z/s_k\Z&\mapr{f_{k}}{}&\Z/s_{k-1}\Z&\mapr{f'_{k-1}}{}&\ldots&\mapr{f'_{2}}{}&\Z/s_{1}\Z\vspace{1ex}\\
&&\mapdown{\beta_{k}}&&\mapdown{\beta_{k-1}}&&&&\mapdown{\beta_{1}}\vspace{1ex}\\
\ldots&\mapr{f''_{k+1}}{}&\Z/d_k\Z&\mapr{f''_{k}}{}&\Z/d_{k-1}\Z&\mapr{f''_{k-1}}{}&\ldots&\mapr{f''_{2}}{}&\Z/d_{1}\Z\vspace{1ex}\\
&&\mapdown{}&&\mapdown{}&&&&\mapdown{}\vspace{1ex}\\
&&0&&0&&&&0
\end{array}
$$
where $f_k(n+s_k\Z)=n+s_{k-1}\Z$, $f'_k$ is the restriction of $f_k$ to $M_k$ and $f''_k(n+d_k\Z)=n+d_{k-1}\Z$.

The columns of the diagram are exact sequences of groups, in other words, the diagram can be interpreted as an exact sequence of inverse systems of abelian groups.

Since inverse limit is a left exact functor, see \cite[Chapter II, Theorem 12.3]{Fuchs}, we obtain an exact sequence
\begin{equation}\label{alg_2}
0\rightarrow \lim\limits_{\leftarrow}M_k\mapr{\alpha}{} \lim\limits_{\leftarrow}\Z/s_k\Z\mapr{\beta}{} \lim\limits_{\leftarrow}\Z/{d_k}\Z
\end{equation}

 The condition \eqref{nowy_2G} yields that the homomorphism $\gamma$ in (\ref{alg_2}) is  surjective, thus we have an exact sequence
\begin{equation}\label{alg_3}
0\rightarrow \lim\limits_{\leftarrow}M_k\mapr{\alpha}{} \lim\limits_{\leftarrow}\Z/{s_k}\Z\mapr{\beta}{} \lim\limits_{\leftarrow}\Z/{d_k}\Z\rightarrow 0
\end{equation}
Indeed, let $(n_k+d_k\Z)_k\in \lim\limits_{\leftarrow}\Z/{d_k}\Z$. By induction we construct the numbers $m_1,m_2,\ldots$ such that $d_k|m_k-n_k$ and $s_k|m_{k+1}-m_k$, for any $k$. Then $\beta((m_k+s_k\Z)_k)=(n_k+d_k\Z)_k$. We set $m_1=n_1$. Assume that $m_1,\ldots m_k$ have been defined. Since $d_k\mid n_{k+1}-n_k$,  $d_k\mid m_{k}-n_k$ and $\gcd(d_{k+1},s_k)=d_k$, there exists integers $x,y$ such that $xd_{k+1}+ys_k=m_k-n_{k+1}$. We set $m_{k+1}=m_k-ys_k$.

There are group isomorphisms $g_k:\Z/{\frac{s_{k}}{d_{k}}}\Z \rightarrow M_k$ given by $g_k(n+\frac{s_{k}}{d_{k}}\Z)=d_kn+s_k\Z$ and making the following diagram commutative
$$
\begin{array}{ccccccccc}
\ldots&\rightarrow &\Z/{ \frac{s_{k+1}}{d_{k+1}}}\Z&\rightarrow& \Z/{\frac{s_{k}}{d_{k}}}\Z& \rightarrow&\ldots &\rightarrow& \Z/{\frac{s_{1}}{d_{1}}}\Z\vspace{1ex}\\
&&\mapdown{g_{k+1}}&&\mapdown{g_k}&&&&\mapdown{g_1}\vspace{1ex}\\
\ldots&\mapr{}{}&M_{k+1}&\mapr{f'_{k+1}}{}&M_{k}&\mapr{f'_{k}}{}&\ldots&\mapr{f'_{2}}{}&M_{1}\vspace{1ex}\\
\end{array}
$$
(the arrows in the upper row represent the canonical projections). It follows that there is an isomorphism
\begin{equation}\label{alg_4}
\lim\limits_{\leftarrow}M_k\cong \lim\limits_{\leftarrow}\Z/{\frac{s_{k}}{d_{k}}}\Z
\end{equation}

By Proposition~\ref{prop:periods} a) it follows that  $\lim\limits_{\leftarrow}M_k$ is isomorphic to $H_{\inn(W)}$. There is an isomorphism given by $\sigma\alpha$, where $\sigma$ is the isomorphism defined in Remark~\ref{limits}.

Now a), b) and c) follow from (\ref{alg_3}), (\ref{alg_4}) and Remark~\ref{limits}. In order to prove d) it is enough to note that $s_k=d_k$ if and only if $s_k\mid c_{k+j}$ for some $j\ge 0$.
\end{proof}

\begin{proof}[Proof of Theorem~\ref{theo:MEF}]
This is an immediate corollary to Proposition~\ref{prop:period}.
\end{proof}

\subsection{Examples}\label{subsec:examples}
\begin{remark}\label{remark:limits}
Given a prime number $p$ and $m\in\Z$ we denote by  $v_p(m)$ be the $p$-valuation of $m$, that is, if $m\neq 0$ then $v_p(m)$ is the maximal integer such that $p^{v_p(m)}\mid m$ and  $v_p(0)=+\infty$. Assume that $t=(t_k)$ is a sequence of natural numbers such that $t_k\mid t_{k+1}$ for any $k$. Set $v_p(t)=\sup_kv_p(t_k)$. The  sequence $t$ yields an inverse system of abelian groups
$$
\ldots\rightarrow Z/t_{k+1}\Z\rightarrow Z/t_{k}\Z \rightarrow\ldots \rightarrow Z/t_{1}\Z
$$
where the arrows represent the canonical projections $n+t_{k+1}\Z\mapsto n+t_{k}\Z$.
The inverse limit $\lim\limits_{\leftarrow}\Z/t_{k}\Z$ of this system is isomorphic to the group
$$
\prod\limits_{p\in\cP}G_p
$$
where $G_p=\Z/p^{v_p(t)}\Z$ if $v_p(t)<+\infty$ and  $G_p=\widehat{\Z}_p$ (the group of $p$-adic numbers) otherwise, i.e. when $\lim_kv_p(t_k)=+\infty$.
\end{remark}

Recall from \eqref{def:Cinf} that
\begin{equation}
\Ainf=\{c\in\N: \forall_{S\subset\cB}\ \exists_{S': S\subseteq S'}: c\in\cA_{S'}\setminus S'\}.
\end{equation}
Our first exaxmple has a finite, non-trivial maximal equicontinuous factor and a finite set $\Ainf$.
\begin{example}
  $\cB=\{36\}\cup \{2p_1,2p_2,\ldots\}\cup \{3q_1,3q_2,\ldots\}$, where $p_1,q_1,p_2,q_2,\ldots$ are pairwise different primes.
Let $S_k=\{36,2p_1,\dots,2p_k,3q_1,\dots,3q_k\}$. Then
\begin{equation*}
s_k=36p_1\cdots p_kq_1\cdots q_k,\; \cA_{S_k}=\{2,3\},\;c_k=d_k=6\ ,
\end{equation*}
so that
\begin{equation*}
\frac{s_k}{d_k}=6p_1\cdots p_kq_1\cdots q_k\ .
\end{equation*}
In particular, the maximal equicontinuous factor of $X_\eta$ is the translation by $1$ on $\Z/6\Z$. Moreover, $\Ainf=\{2,3\}$, so that $\emptyset\neq\overline{\inn(W)}\neq W$
by Theorems~\ref{theo:minimality} and~\ref{theo:proximal}.
\end{example}
Our next example has an infinite maximal equicontinuous factor different from $H$ and an infinite set $\Ainf$.
\begin{example}
Let $p_1,q_1,p_2,q_2,\ldots$ be pairwise different primes. Let
$$
\cB=\cB_1\cup\cB_2\cup\cB_3\ldots
$$
where
$$
\begin{array}{l}
\cB_1=\{p_1q_1\}\\
\cB_2=\{p_1p_2^2, p_1q_2^2, q_1q_2^2\}\\
\cB_3=\{p_1p_2p_3^2,p_1p_2q_3^2,p_1q_2q_3^2, q_1q_3^2\}\\
\cB_4=\{p_1p_2p_3p_4^2,p_1p_2p_3q_4^2, p_1p_2q_3q_4^2, p_1q_2q_4^2, q_1q_4^2\}\\
\cB_5=\{p_1p_2p_3p_4p_5^2,p_1p_2p_3p_4q_5^2,p_1p_2p_3q_4q_5^2,p_1p_2q_3q_5^2, p_1q_2q_5^2, q_1q_5^2\}\\
\ldots
\end{array}$$
That is,
$$
\cB_{k+1}=\{p_1\ldots p_k p_{k+1}^2,\;p_1\ldots p_k q_{k+1}^2,\;p_1\ldots p_{k-1} q_k q_{k+1}^2\}\cup\{\frac{b q_{k+1}^2}{q_k^2}:b\in\cB_k\setminus\{p_1\ldots p_{k-1}p_{k}^2,p_1\ldots p_{k-1 }q_{k}^2\}\}
$$
for $k\ge 2$.

Let $S_k=\cB_1\cup\ldots\cup\cB_k$. Then $s_k=\lcm(S_k)=p_1p_2^2\ldots p_k^2q_1q_2^2\ldots q_k^2$ and
$$
\cA_{S_k}=S_k\cup\{p_1\ldots p_k,\;p_1\ldots p_{k-1}q_k,\; p_1\ldots p_{k-2}q_{k-1},\;\ldots, p_1q_2,q_1\}\ ,
$$
so that
\begin{equation*}
\prim\cA_{S_k}
=
\{p_1\ldots p_k,\;p_1\ldots p_{k-1}q_k,\; p_1\ldots p_{k-2}q_{k-1},\;\ldots, p_1q_2,q_1\}\ .
\end{equation*}
Hence
\begin{equation*}
c_k=p_1\cdots p_k q_1\cdots q_k\quad\text{and}\quad d_k=\gcd(s_k,c_{k+j})=c_k\ ,
\end{equation*}
so that
\begin{equation*}
\frac{s_k}{d_k}=p_2\cdots p_kq_2\cdots q_k\ .
\end{equation*}
Hence $H_{\inn(W)}\cong\prod_{i=2}^{+\infty}\Z/p_iq_i\Z$ and $H/H_{\inn(W)}\cong\prod_{i=1}^{+\infty}\Z/p_iq_i\Z$ are infinite compact groups.
Moreover,
$$\Ainf=\limsup_{k\to\infty}\cA_{S_k}\setminus S_k=\{q_1,p_1q_2,p_1p_2q_3,p_1p_2p_3q_4,\ldots\}$$ is infinite and does not contain the number $1$, thus $\emptyset\neq\overline{\inn(W)}\neq W$ by Theorems~\ref{theo:minimality} and~\ref{theo:proximal}.
\end{example}
We end with a non-trivial example where the maximal equicontinuous factor equals $H$ and $\Ainf$ is an infinite set.
\begin{example}
Let $q, p_1,p_2,\ldots$ be pairwise different odd primes. Let
$$
\cB=\cB_1\cup\cB_2\cup\cB_3\ldots
$$
where
$$
\begin{array}{l}
\cB_1=\{p_1q\}\\
\cB_2=\{p_2q,p_1p_2\}\\
\cB_3=\{p_3q,p_1p_3, p_2p_3\}\\
\cB_4=\{p_4q,p_1p_4,p_2p_4,p_3p_4\}\\
\ldots
\end{array}$$
That is,
$$
\cB_{k}=\{p_kq,p_1p_k,\ldots,p_{k-1}p_k\}
$$
for $k\ge 1$.
Let $S_k=\cB_1\cup\ldots\cup\cB_k$. Then $s_k=\lcm(S_k)=qp_1\ldots p_k$ and
$$
\cA_{S_k}=S_k\cup\{p_1,\ldots, p_k\}\cup \{q\}
$$
hence $c_{\cA_{S_k}}=qp_1\ldots p_k=\lcm(S_k)$,
so that $s_k=c_k=d_k$ for all $k$. In particular
$\inn(W)$ is aperiodic by Proposition~\ref{prop:periods}.
Moreover, $$\Ainf=\limsup_{k\to\infty}\left(\cA_{S_k}\setminus S_k\right)=\{q,p_1,p_2,\ldots\}$$ is infinite and does not contain the number $1$, thus $\emptyset\neq\overline{\inn(W)}\neq W$
by Theorems~\ref{theo:minimality} and~\ref{theo:proximal}.
\end{example}


\begin{thebibliography}{99}
\bibitem{Ab-Le-Ru}
 H. Abdalaoui, M. Lema\'nczyk, T. de la Rue, {\em A dynamical
point of view on the set of $\mathcal{B}$-free integers},  International
 Mathematics Research Notices {\bf 16} (2015), 7258-7286.

\bibitem{BaakeHuck14}
M.~Baake and C.~Huck.
\emph{Ergodic properties of visible lattice points}.
Proc. Steklov Inst. Math. \textbf{288} (2015), 165--188.

\bibitem{BHS2015}
M.~Baake, C.~Huck, and N.~Strungaru.
\emph{On weak model sets of extremal density}. Preprint (2015)
arXiv:1512.07129v2, Indagationes Mathematicae (in press).

\bibitem{BKKL2015}
A.~Bartnicka, S.~Kasjan, J.~Ku\l{}aga-Przymus, and M.~Lema\'n{}czyk.
\emph{${\cB}$-free sets and dynamics}.
Preprint (2015) ArXiv 1509.08010. To appear in Trans. Amer. Math. Soc.

\bibitem{Behrend1948}
F.A.~Behrend.
\emph{Generalization of an inequality of Heilbronn and Rohrbach.}
Bull. Amer. Math. Soc. \textbf{54} (1948), 681--684.

\bibitem{DE1936}
H.~Davenport and P.~Erd\"os.
\emph{On sequences of positive integers.}
Acta Arithmetica \textbf{2} (1936), 147--151.

\bibitem{DE1951}
H.~Davenport and P.~Erd\"os.
\emph{On sequences of positive integers.}
J. Indian Math. soc. (N.S.) \textbf{15} (1951), 19--24.

\bibitem{Do}
 T. Downarowicz, {\em Survey of odometers and Toeplitz flows},
 Contemporary
 Mathematics, Algebraic and Topological Dynamics (Kolyada, Manin, Ward
 eds.), vol. 385, 2005, pp. 7-38.

\bibitem{Fuchs}  {\sc L. Fuchs}, {\em Infinite abelian groups.} Vol. I. Pure and Applied Mathematics, Vol. 36 Academic Press, New York-London, 1970.

\bibitem{hall-book}
R.~R.~Hall, Sets of multiples, vol. 118 of Cambridge Tracts in Mathematics, Cambridge University Press, Cambridge, 1996.

\bibitem{MR687978}
{\sc H.~Halberstam and K.~F. Roth}, {\em Sequences}, Springer-Verlag, New
  York-Berlin, second~ed., 1983.

\bibitem{JK1969}
K.~Jacobs, M.~Keane,
\emph{0-1-sequences of Toeplitz type}.
Zeitschrift f\"ur Wahrscheinlichkeitstheorie und Verwandte Gebiete (Prob. Th. Rel. Fields)
\textbf{13} (1969), 123--131.

\bibitem{KR2015}
G.~Keller, C.~Richard, \emph{Dynamics on the graph of the torus parametrisation,} Preprint (2015) ArXiv 1511.06137. Ergod. Th.\&
Dynam. Sys. first online, doi:10.1017/etds.2016.53

\bibitem{KR2016}
G.~Keller, C.~Richard, \emph{Periods and factors of weak model sets,} Preprint (2017).

\bibitem{Peckner2012}
R.~Peckner.
\emph{Uniqueness of the measure of maximal entropy for the squarefree flow}.
Israel J. Math. \textbf{210} (2015), 335--357.

\bibitem{Sa}
 P. Sarnak, {\em Three lectures on  M\"obius function, randomness and
 dynamics}, publications.ias.edu/sarnak/.

\end{thebibliography}
\end{document}